\documentclass[12pt]{article}

\usepackage{graphicx}	
\usepackage{epstopdf}
\usepackage{latexsym}
\usepackage{amssymb}
\usepackage{amsmath}
\usepackage[all]{xy}
\usepackage{hyperref}

\newtheorem{definition}{Definition}
\newtheorem{remark}{Remark}

\newtheorem{proposition}{Proposition}
\newtheorem{corollary}{Corollary}

\newtheorem{theorem}{Theorem}
\newtheorem{lemma}{Lemma}

\newcommand{\be}[1]{\begin{equation}\label{eq:#1}}
\newcommand{\ee}{\end{equation}}

\newcommand{\void}[1]{}

\def\N {\mathbb{N}}
\def\Z {\mathbb{Z}}

\def\C {\mathbb{C}}

\def\id {\mathrm{id}}

\def\Obj {\mathrm{Obj}}
\def\Mor {\mathrm{Mor}}
\def\Hom {\mathrm{Hom}}
\def\oti {\otimes}

\def\End {\mathrm{End}}
\def\2cat {\mathcal{K}}
\def\PoincDual {\xymatrix{&\ar@{|->}[l] \ar@{|->}[r]&}}

\def\cC {\ensuremath{\mathcal{C}}}

\def\one {{1\!\!1}}

\def\im {\mathrm{Im}}

\def\extS {\ensuremath{\mathrm{X}}}
\def\eCob {\ensuremath{\mathrm{M}}}
\def\bl {\ensuremath{\mathcal{H}}}
\def\F {\ensuremath{\mathcal F}}

\begin{document}

\title{Reducibility of quantum representations of mapping class groups}

\author{J\o rgen Ellegaard Andersen and Jens Fjelstad}

\maketitle

\begin{abstract}In this paper we provide a general condition for the reducibility of the Reshetikhin-Turaev quantum 
representations of the mapping class groups. Namely, for any modular tensor category with a special symmetric Frobenius algebra
with a non-trivial genus one partition function, we prove that the quantum representations of all the mapping class groups 
built from the modular tensor category are reducible. In particular for $SU(N)$ we get reducibility
for certain levels and ranks. For the quantum $SU(2)$ Reshetikhin-Turaev theory we construct a decomposition for all even levels. We conjecture
this decomposition is a complete decomposition into irreducible representations for high enough levels.
\end{abstract}

\newpage

One of the features of three dimensional TQFT, as defined first by Witten, Atiyah and Segal in \cite{W1}, \cite{Atiyah} and \cite{Segal}, and
 further made precise by Reshetikhin and Turaev in \cite{RT1}, \cite{RT2}, is that it provides finite dimensional representations of mapping class groups of compact orientable surfaces, possibly with marked points. More precisely, a TQFT based on a modular category $\cC$ with ground field $k$ associates finite dimensional $k$-vector spaces $\bl(\extS)$ to surfaces $\extS$, and linear isomorphisms $\rho(f):\bl(\extS)\rightarrow\bl(\extS')$ to orientation preserving homeomorphisms $f:\extS\rightarrow \extS'$ depending only on the mapping class $[f]$. The assignment $f\mapsto \rho(f)$ is quasi-functorial, where the failure of functoriality is measured by a non-zero multiplicative factor, with the end result that $\rho:\mathrm{MC}(\extS)\rightarrow\End(\bl(\extS))$ becomes a projective representation. See \cite{T1} for a complete treatment of this.

Another machine that produces finite dimensional projective representations of mapping class groups is rational conformal field theory \cite{MS, FrS}. Furthermore, given a rational CFT with chiral algebra $\mathcal{V}$, it is expected that the corresponding representation of $\mathrm{MC}(\extS)$ is isomorphic to the representation given by a TQFT based on the modular category $\mathrm{Rep}(\mathcal{V})$. With some assumptions on $\mathcal{V}$ the statement has been shown to be true for genus $0$ and genus $1$, but is still open for higher genus.

Apart from some special cases, not much is known about the (ir-)reducibility these representations. For the rest of the paper we focus on surfaces with no marked points.\\

\noindent
In the Reshetikhin-Tureav TQFT for $U_q(sl(2,\C))$ Roberts has shown, with the use of the skein theoretical construction of this Resetikhin-Turaev TQFT by Blanchet, Habegger, Masbaum and Vogel \cite{BHMV1}, \cite{BHMV2}, that for $k+2$ prime the representations are irreducible for any genus $g\geq 1$~\cite{Rob}. We recall that this result played a key role in proving that the mapping class groups does not have Kazhdan's Property T \cite{A2}.

\noindent
The method used to classify modular invariant torus partition functions in the $SU(2)$ WZW models of CFT presented in \cite{GepQ} confirm this result for genus $1$, since it is shown that the commutant of the representation is trivial when $k+2$ is prime. It is furthermore shown that the commutant is non-trivial for all other (integer) values of $k>1$, and the corresponding representations are reducible.\\

\noindent
The TQFT representations of the mapping class group of a surface of genus $1$ (i.e. of $SL(2,\Z)$) are rather special since they factor through the finite group $SL(2,\Z/N\Z)$ \cite{CosteGannon,Bantay}. It has been shown in some detail \cite{Kerler} how the representations decompose for $g=1$ in the $SU(3)$ theory at level $k$ with $k+3$ prime and $k+3\equiv 2 \mathrm{mod} 3$. Furthermore the same reference contains the result that the representations in the $SU(N)$ theory at level $k$ where $N>2$, $k>N$, and with $k+N$ and $N$ coprime are reducible for all genus $g\geq 1$.\\

\noindent
A rational CFT with chiral algebra $\mathcal{V}$ is defined by a so called symmetric special Frobenius algebra in the modular category $\mathrm{Rep}(\mathcal{V})$. By means of the Frobenius algebra, a CFT assigns to a closed oriented surfaces $\extS$, elements in $\bl(\extS)\otimes_\C\bl(\extS)^*$ (correlators) commuting with the TQFT representation of $\mathrm{MC}(\extS)$ on $\bl(\extS)$. The most familiar of these correlators are the modular invariant torus partition functions, which are elements of $\bl(T^2)\otimes_\C\bl(T^2)^*$. In a (canonically given) basis $\{\chi_i\}_i$ of $\bl(T^2)$, torus partition functions are represented by matrices $Z(A)=(Z_{ij}(A))$, determined by the Frobenius algebra $A$.
We say that this partition function is trivial, if $Z(A)$ is proportional to the the identity. We use techniques from rational CFT to investigate the TQFT representations of mapping class groups of compact closed orientable surfaces $\extS_g$ of genus $g\geq 1$, without marked points.

The main results are presented in two theorems:

\begin{theorem}
Let \cC{} be a modular tensor category.
If there exists a special symmetric Frobenius algebra $A$ in \cC{} such that $Z(A)$ is not trivial, then the TQFT representation of the mapping class group of $\extS_g$ on $\bl(\extS_g)$ is reducible for every $g\geq 1$.
\label{thm:reducibility}
\end{theorem}

\begin{remark}
We expect that the qualifier ''$Z(A)$ is not trivial'' can be replaced by requiring that the symmetric special Frobenius algebra not Morita equivalent to $\one$. As we will see (cf remark~\ref{rem:morita}) our present proof does not allow us to do this. It should be mentioned, however, that there is no known example of an algebra that yields the diagonal modular invariant torus partition function, but at the same time is not Morita equivalent to $\one$.
\end{remark}
If we include in the set of symmetric special Frobenius algebras in modular categories also examples arising as canonical endomorphisms in nets of type $III_1$ subfactors on the circle, we get the following

\begin{corollary}
The TQFT representations are reducible for all genus $g\geq 1$ in the theories based on 

$SU(2)$ for all levels $k\in 2\N$, $k\geq 4$

$SU(3)$ for all levels $k\in \N$, $k\geq 3$

$SU(N)$ for $N=mq$, $m,q\in\N$, $m>1$ for level $k$ such that $kq\in 2 m\N$ if $N$ is even and $kq\in m\N$ if $N$ is odd.

\end{corollary}
The list is far from exhaustive, but contains the most interesting examples from the point of view of TQFT. Some other examples are listed in section~\ref{sec:reducibility}, the largest and most important class of which is all theories admitting simple current invariants~\cite{FRS3}. It is generally believed that most non-trivial modular invariants will be simple current invariants, so although some of the other invariants may be spurious (i.e. they will not arise as the torus partition function of any CFT, and thus won't be realised by a symmetric special Frobenius algebra), the theorem most likely implies that reducibility for all genus $g\geq 1$ is a rather generic feature of TQFT.\\

\noindent
In the $SU(2)$ case, the modular invariant torus partition functions have an ADE classification \cite{CaItZu}. In addition, the results in \cite{KO,Ostrik} can be said to give an ADE classification of the (Morita equivalence classes of) symmetric special Frobenius algebras in the modular categories $\cC_k$ of integrable dominant highest weight representations of the untwisted affine Lie algebra $\widehat{\mathfrak{su}}(2)$ at level $k$. These results are used to produce projectors $\Pi_\pm^{g,k,D/E}$ on the corresponding spaces $\bl(\extS_g)$, from which we can derive the following theorem.
 \begin{theorem}
 	Let $V^{g,k}\equiv(\bl(\extS_g),\rho_g)$ denote the representation of the mapping class group of $\extS_g$ given by the TQFT based on the modular category $\cC_k$, i.e. the Reshetikhin-Turaev TQFT for $U_q(sl(2,\C))$. If $k\in 2\Z_+$, $k\geq 4$, then have the following decomposition:
	\begin{align}
		V^{g,k}& =V^{g,k,D}_+\oplus V^{g,k,D}_-\equiv\im(\Pi^{g,k,D}_+)\oplus\im(\Pi^{g,k,D}_-).
	\end{align}
For $k=10, 16$, and $28$ we have further
	\begin{align}
		V^{g,k}& =V^{g,k,E}_+\oplus V^{g,k,E}_-\equiv\im(\Pi^{g,k,E}_+)\oplus\im(\Pi^{g,k,E}_-)\\
		&=V^{g,k,E}_+\oplus W^{g,k}\oplus V^{g,k,D}_-\equiv\im(\Pi^{g,k,E}_+)\oplus\im(\Pi^{g,k,D}_+\Pi^{g,k,E}_-)\oplus\im(\Pi^{g,k,D}_-).
	\end{align}
For any $g\geq 1$ and for any $k$ in the given range, the vector spaces underlying $V^{g,k,D/E}_\pm$, $W^{g,k}$ are non-zero.
\label{thm:su2}
\end{theorem}

\noindent
We find it very interesting to understand if this decomposition is further reducible or not for genus $g\geq 2$. In fact for high enough level, we conjecture this decomposition can not be further decomposed. 
 
\noindent
We also find it very interesting to try to understand the above decompositions in terms of the geometric model (see e.g. \cite{ADW}, \cite{H1}, \cite{A1}, \cite{A3} and references in there) for these TQFT's. In particular we expect that the above decomposition agrees with the one found in \cite{AM}. Furthermore, the higher genus zero decompositions given by Blanchet in \cite{B1} might also be related to the decompositions obtainable from the techniques introduced in this paper.

\section{Preliminaries}\label{sec:prel}
In this section we gather some preliminary results, and establish conventions and notation for working in a ribbon category. The conventions and notation coincide to large extent with those of references \cite{FRS1,FrFRS,FjFRS1,FjFRS2}, which we refer to for more extensive discussions.

\subsection{Conventions and Notation for Ribbon Categories}\label{ssec:conventions}
Let \cC{} be a strict abelian $\C$-linear semisimple ribbon category with a finite number of isomorphism classes of simple objects. In practice we will be interested in modular categories, but the maximal non-degeneracy of the braiding will not play any role in this paper\footnote{Note that a modular category in this sense is more restrictive than in the original sense, but includes all modular tensor categories obtained in rational conformal field theory.}.
With this definition, \cC{} is automatically idempotent complete (Karoubian), i.e. every idempotent is split: let $p\in\End(U)$ be an idempotent, $p\circ p=p$, then $p$ is called split if there is a triple $(V,e,r)$ with $V$ an object, $e\in\Hom(V,U)$ a monic, and $r\in\Hom(U,V)$ a morphism s.t. $r\circ e=\id_{V}$ and $e\circ r=p$, i.e. $(V,e,r)$ is a {\em retract}. The monic $e$ is $\im(p)$ in the categorical sense, and by abuse of notation we say $\im(p)=V$.
We indicate that $V$ is a subobject of $U$ by writing $V\prec U$, and the morphisms of a retract $(V,e,r)$ of $U$ are denoted $e_{V\prec U}$ respectively $r_{V\prec U}$. \\

We choose a representative of every isomorphism class of simple objects and denote these by $U_i$, $i\in I$, where $U_0=\one$, the tensor unit.
The braiding isomorphisms are denoted $c_{U,V}\in\Hom(U\oti V,V\oti U)$. The twist isomorphisms are $\theta_U\in\End(U)$, and for a simple object $U_i$ we write $\theta_{U_i}=\theta_i\id_{U_i}$ where $\theta_i\in\C^\times$. Evaluation and coevaluation morphisms of the right duality are denoted by $d_U$ and $b_U$ respectively, while those for the left duality are denoted $\tilde{d}_U$ resp. $\tilde{b}_U$.\\

We will make frequent use of graphical calculus in \cC, let us therefore fix some conventions regarding diagrams representing morphisms. All diagrams are to be read upwards, from bottom to top.
Morphisms $e$ and $r$ of a retract are drawn as

\begin{figure}[h]
	\centering
	\begin{picture}(100,70)(0,0)
		\put(0,0)		{\includegraphics[scale=0.4]{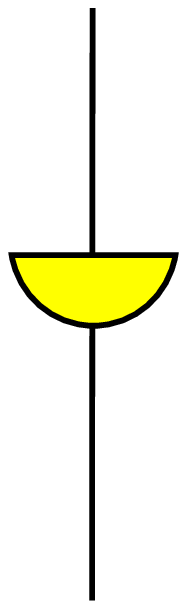}}
		\put(8,34)		{\scriptsize $e$}
		\put(120,0)	{\includegraphics[scale=0.4]{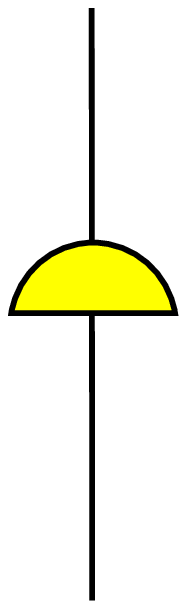}}
		\put(128,35)	{\scriptsize $r$}
		\put(6,-6)		{\scriptsize $V$}
		\put(126,-6)		{\scriptsize $U$}
		\put(6,70)		{\scriptsize $U$}
		\put(126,70)		{\scriptsize $V$}
	\end{picture}
\end{figure}
The various structural morphisms are pictured as
\begin{figure}[h]
	\centering
	\begin{picture}(200,80)
		\put(-30,20)	{$b_U=$}
		\put(-3,50)		{\scriptsize $U$}
		\put(28,50)	{\scriptsize $U^\vee$}
		\put(0,0)		{\includegraphics[scale=0.4]{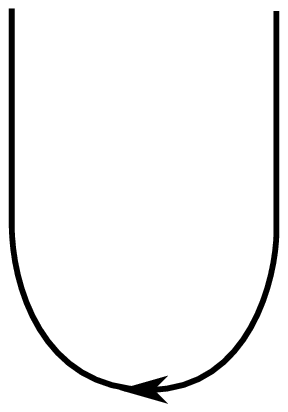}}
		\put(45,20)	{ $d_U=$}
		\put(77,-10)	{\scriptsize $U^\vee$}
		\put(108,-10)	{\scriptsize $U$}
		\put(80,0)		{\includegraphics[scale=0.4]{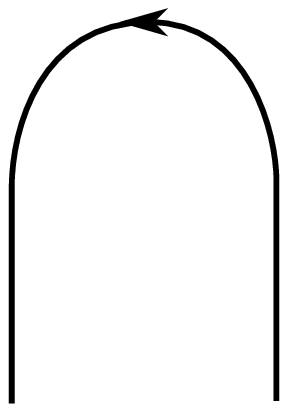}}
		\put(120,30)	{$c_{U,V}=$}
		\put(152,-10)	{\scriptsize $U$}
		\put(180,-10)	{\scriptsize $V$}
		\put(152,63)	{\scriptsize $V$}
		\put(180,63)	{\scriptsize $U$}
		\put(155,0)		{\includegraphics[scale=0.35]{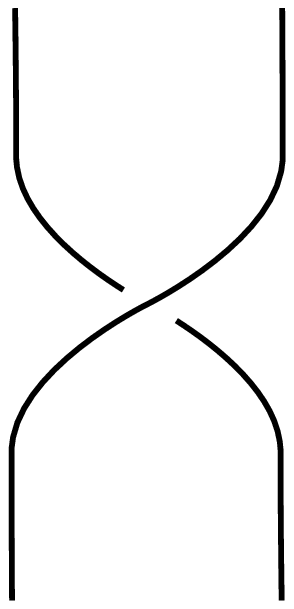}}
		\put(220,30)	{$\Theta_{U}=$}
		\put(253,-10)	{\scriptsize $U$}
		\put(253,63)	{\scriptsize $U$}
		\put(255,0)	{\includegraphics[scale=0.35]{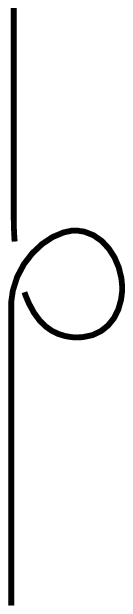}}
	\end{picture}
\end{figure}

with the obvious modifications for inverses of the isomorphisms and evaluation and coevaluation for the left duality.
We fix once and for all bases $\{\lambda_{(ij)k}^\alpha\}_\alpha$ of morphism spaces $\Hom(U_i\oti U_j,U_k)$ and dual bases  $\{\bar{\lambda}^{(ij)k}_{\bar{\alpha}}\}_{\bar{\alpha}}$ of $\Hom(U_k,U_i\oti U_j)$.
\begin{figure}[h]
	\centering
	\begin{picture}(400,100)
		\put(40,50)	{$\lambda_{(ij)k}^\alpha$}
		\put(104,52)	{$\alpha$}
		\put(70,50)	{$=$}
		\put(100,0)	{\includegraphics[scale=0.3]{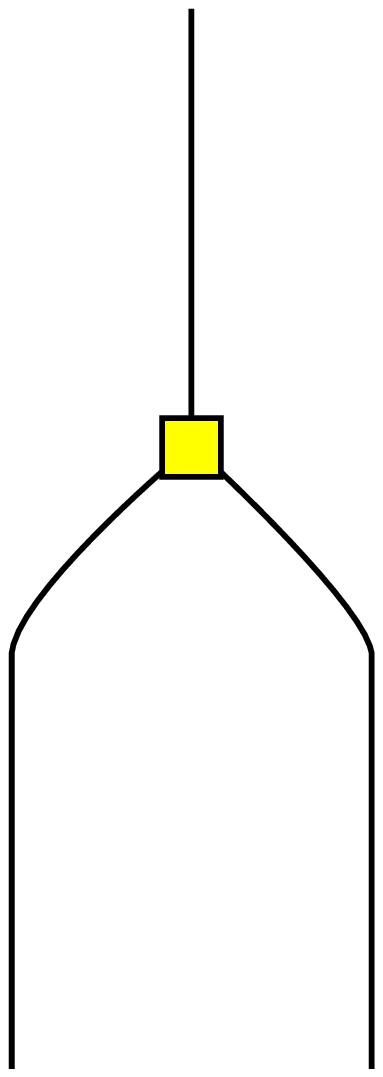}}
		\put(97,-8)	{\scriptsize $U_i$}
		\put(130,-9)	{\scriptsize $U_j$}
		\put(112,95)	{\scriptsize $U_k$}
		\put(240,50)	{$\bar{\lambda}^{(ij)k}_{\bar{\alpha}}$}
		\put(270,50)	{$=$}
		\put(312,-8)	{\scriptsize $U_k$}
		\put(298,95)	{\scriptsize $U_i$}
		\put(330,95)	{\scriptsize $U_j$}
		\put(298,36)	{$\bar{\alpha}$}
		\put(300,0)	{\includegraphics[scale=0.3]{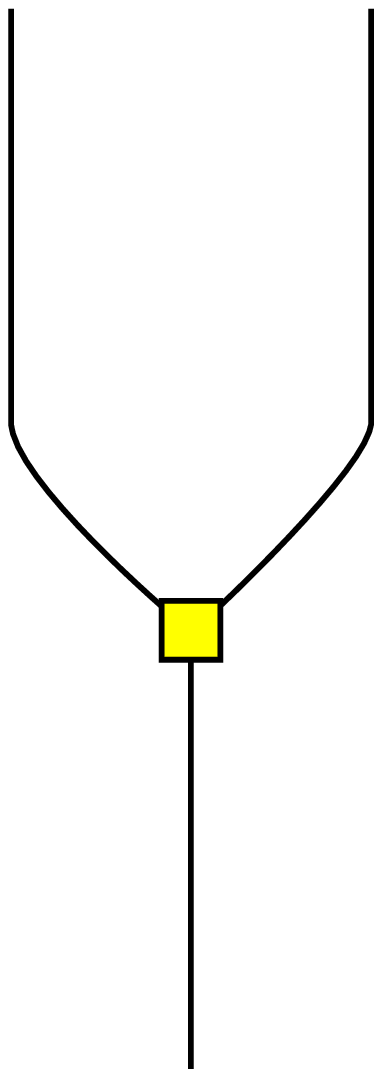}}
	\end{picture}
\end{figure}

\noindent
The meaning of dual bases is shown in the following figure.

	\begin{picture}(300,120)(-150,-10)
		\put(0,0)		{\includegraphics[scale=0.2]{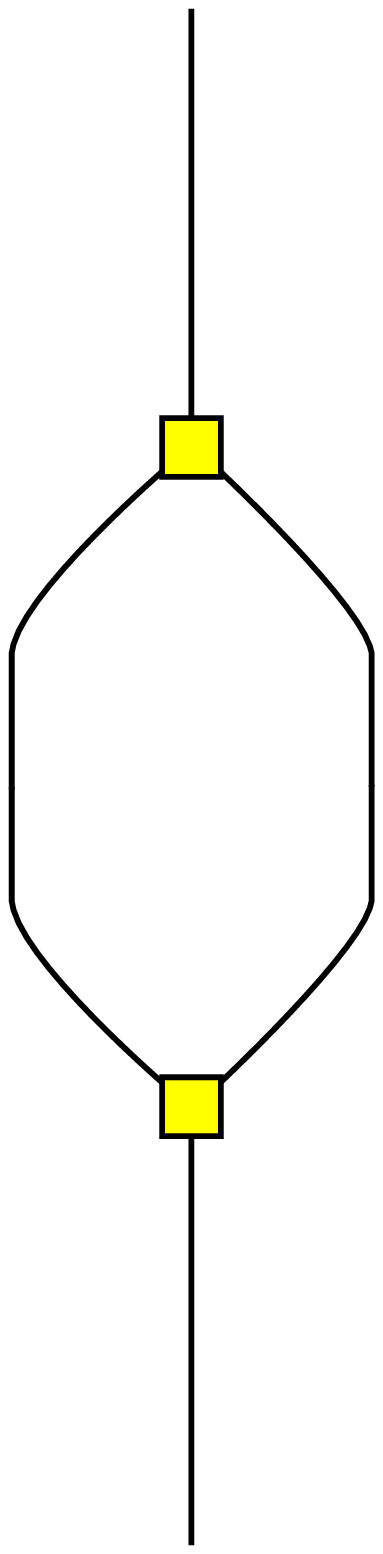}}
		\put(50,40)		{$=$}
		\put(100,0)	{\includegraphics[scale=0.4]{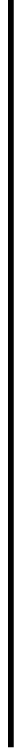}}
		\put(8,-8)		{\scriptsize $U_i$}
		\put(8,91)		{\scriptsize $U_i$}
		\put(-10,40)		{\scriptsize $U_j$}
		\put(25,40)	{\scriptsize $U_k$}
		\put(0,23)		{\small$\bar\alpha$}
		\put(0,60)		{\small$\beta$}
		\put(67,40)	{$\delta_{\alpha,\beta}$}
		\put(98,-8)		{\scriptsize $U_i$}
		\put(98,91)		{\scriptsize $U_i$}
	\end{picture}

When one of the simple objects involved is the tensor unit we further choose the morphisms as shown in figure~\ref{fig:adaptedbasis},
\begin{figure}
	\centering
	\begin{picture}(400,100)
	\put(40,50)		{$=$}
	\put(90,50)		{$=$}
	\put(-2,-8)			{\scriptsize $U_i$}
	\put(30,-8)		{\scriptsize $\one$}
	\put(13,95)		{\scriptsize $U_i$}
	\put(70,-8)		{\scriptsize $U_i$}
	\put(70,95)		{\scriptsize $U_i$}
	\put(112,-8)			{\scriptsize $\one$}
	\put(145,-8)		{\scriptsize $U_i$}
	\put(128,95)		{\scriptsize $U_i$}
	\put(250,0)		{\includegraphics[scale=0.3]{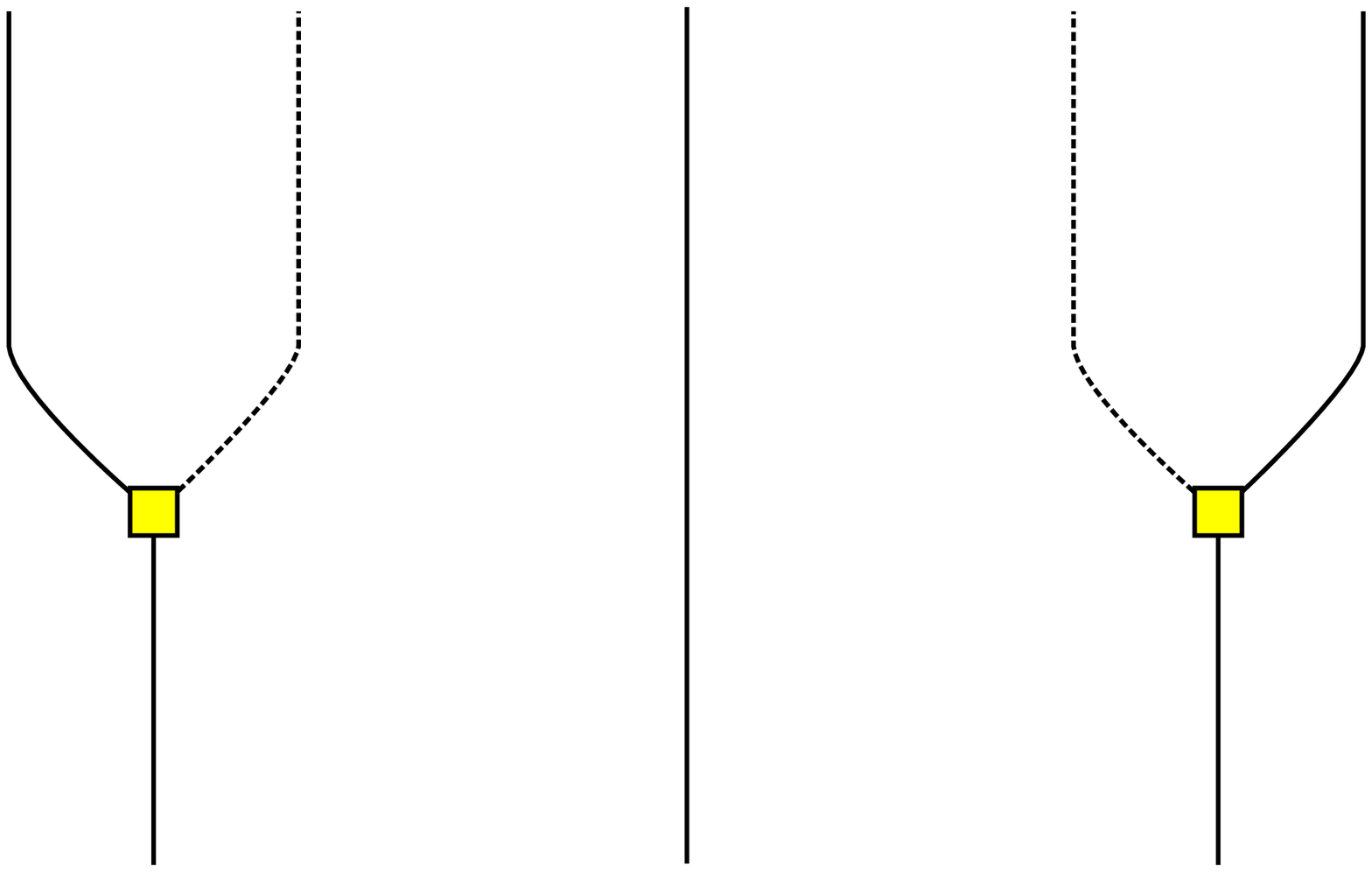}}
	\put(290,50)		{$=$}
	\put(340,50)		{$=$}
	\put(263,-8)		{\scriptsize $U_i$}
	\put(248,95)		{\scriptsize $U_i$}
	\put(280,95)		{\scriptsize $\one$}
	\put(320,-8)		{\scriptsize $U_i$}
	\put(320,95)		{\scriptsize $U_i$}
	\put(378,-8)		{\scriptsize $U_i$}
	\put(363,95)		{\scriptsize $\one$}
	\put(393,95)		{\scriptsize $U_i$}
	\put(0,0)			{\includegraphics[scale=0.3]{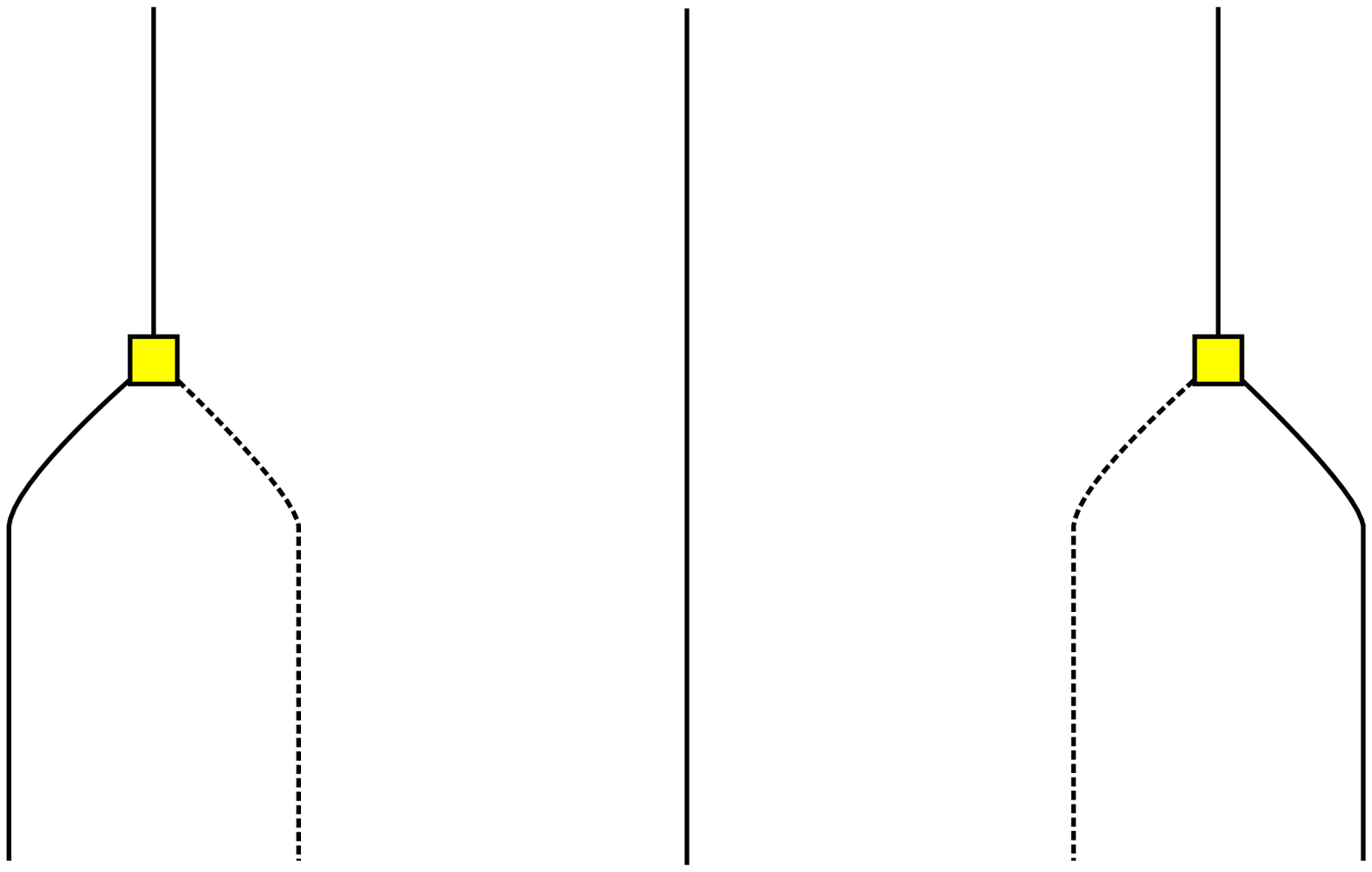}}
	\end{picture}
	\caption{Choice of adapted basis.}
		\label{fig:adaptedbasis}
\end{figure}
which is possible since \cC{} is strict.

\paragraph{A category related to $\widehat{\mathfrak{su}}_k(2)$}~\\
The category of integrable highest weight modules of the untwisted affine algebra $\widehat{\mathfrak{su}}_k(2)$ for some $k\in\N$ or, alternatively, the semisimple part of the category of representations of $U_q(\mathfrak{sl}_2)$ for $q=e^{\frac{2\pi i}{k+2}}$, is a modular tensor category. In CFT it is the category underlying the $\mathfrak{su}(2)$ WZW model at level $k$.
We define \cC$_k${} to be strictifications of these categories. There are $k+1$ isomorphism classes of simple objects, i.e. $I=\{0,1,\ldots ,k\}$, with fusion rules given by
\be{su2fusion}
	U_i\oti U_j\cong \oplus_{l=|i-j|}^{\mathrm{min}\{i+j,2k-i-j\}}U_l.
\ee
In particular we have the fusion coefficients $N_{ij}^{\phantom{ij}l}\in\{0,1\}$ for all values of $i$, $j$, and $l$.
The twist coefficients of the simple objects are given by
\be{su2twist}
	\theta_j=q^{\frac{j(j+2)}{4}}
\ee
and the quantum dimensions by
\be{su2qdim}
	d_j\equiv\mathrm{dim}(U_j)=\frac{q^{j/2}-q^{-j/2}}{q^{1/2}-q^{-1/2}}.
\ee
The object $U_k$ is invertible\footnote{In physics nomenclature an invertible object is known as a \em{simple current}.}, $U_k\oti U_k\cong\one$, a fact that will play an important role in the next section. We note that $d_k=1$ and $\theta_k=i^k$, so $\theta_k=1$ precisely when $k=4N$ and $\theta_k=-1$ precisely when $k=2(2N+1)$.
One important feature of invertible objects is that, tensoring any simple object with an invertible object gives a simple object. In particular, the Picard group $Pic(\cC)$ of \cC{}, i.e. the multiplicative group generated by isomorphism classes of invertible objects, acts on isomorphism classes of simple objects (though not in general on the set of simple objects).
Obviously, $Pic(\cC_k)\cong\Z_2$.

\subsection{Algebras in tensor categories and an Endofunctor}\label{ssec:algebras}

By an algebra in \cC{} we shall mean an associative unital algebra, i.e. an object $A$, a multiplication morphism $m: A\otimes A\rightarrow A$ and a unit morphism $\eta:\one\rightarrow A$ satisfying the associativity and unit constraints. The notion of a (coassociative, counital) coalgebra is the obvious dual concept, with comultiplication and counit denoted by $\Delta$ and $\varepsilon$ respectively. Graphically, these morphisms and constraints are pictured as in figure~\ref{fig:algcoalg}.
\begin{figure}[h]
	\centering
	\begin{picture}(400,250)(0,0)
		\put(70,150)		{\includegraphics[scale=0.5]{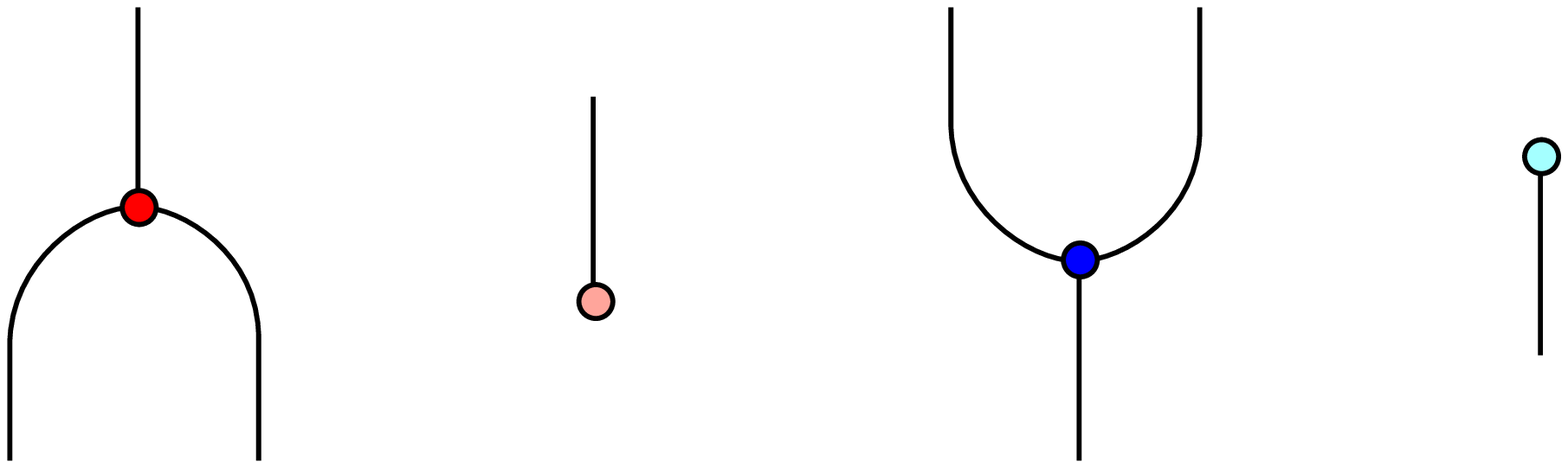}}
		\put(50,190)		{$m=$}
		\put(67,143)		{\scriptsize $A$}
		\put(109,143)		{\scriptsize  $A$}
		\put(89,230)		{\scriptsize $A$}
		\put(140,190)		{$\eta =$}
		\put(165,217)		{\scriptsize $A$}
		\put(195,190)		{$\Delta =$}
		\put(224,230)		{\scriptsize $A$}
		\put(266,230)		{\scriptsize $A$}
		\put(246,143)		{\scriptsize $A$}
		\put(295,190)		{$\varepsilon =$}
		\put(323,160)		{\scriptsize $A$}
		\put(60,70)		{\includegraphics[scale=0.25]{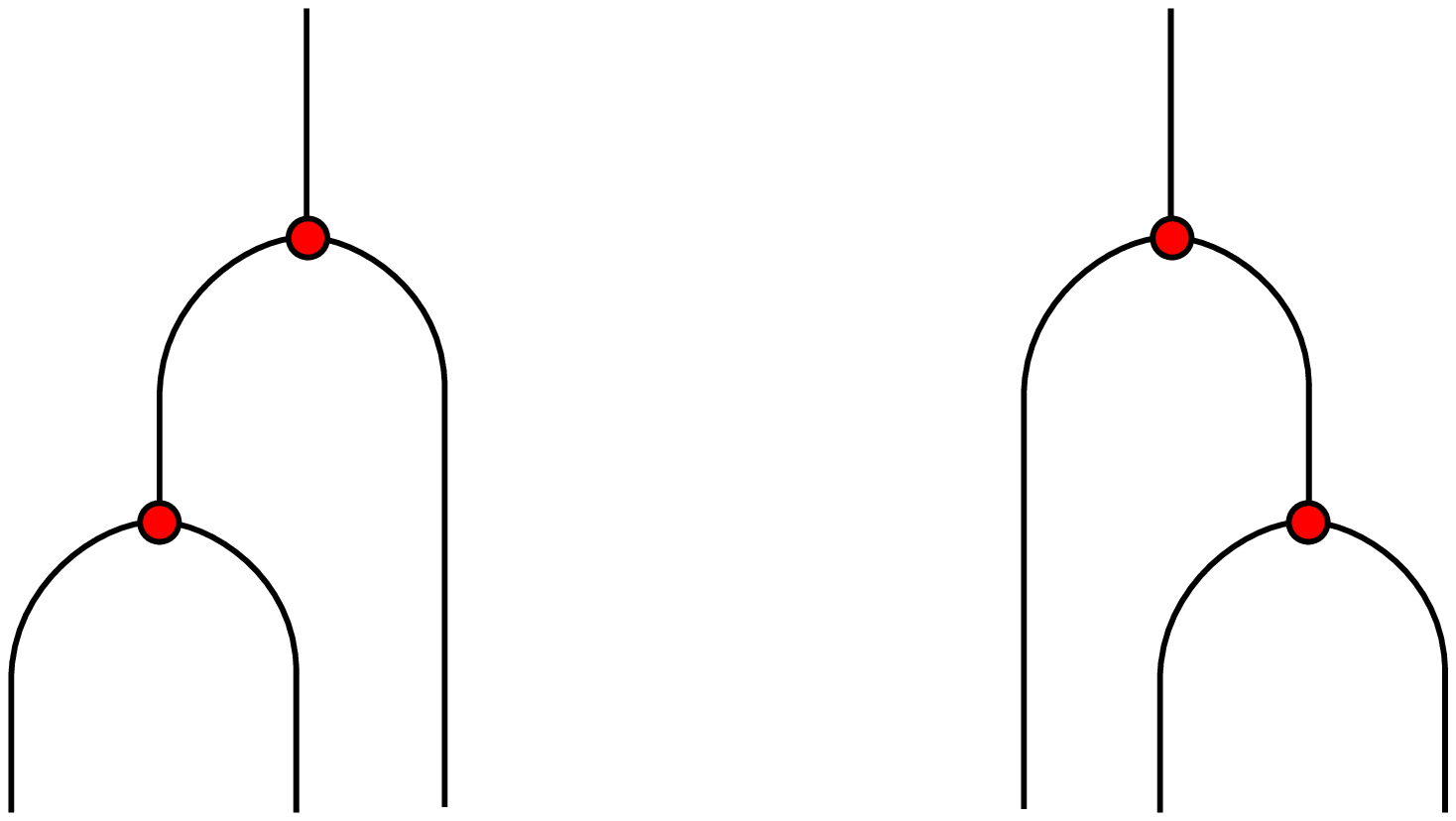}}
		\put(107,95)		{$=$}
		\put(107,25)		{$=$}
		\put(275,90)		{\scriptsize $=$}
		\put(303,90)		{\scriptsize $=$}
		\put(275,20)		{\scriptsize $=$}
		\put(303,20)		{\scriptsize $=$}
		\put(250,70)	{\includegraphics[scale=0.25]{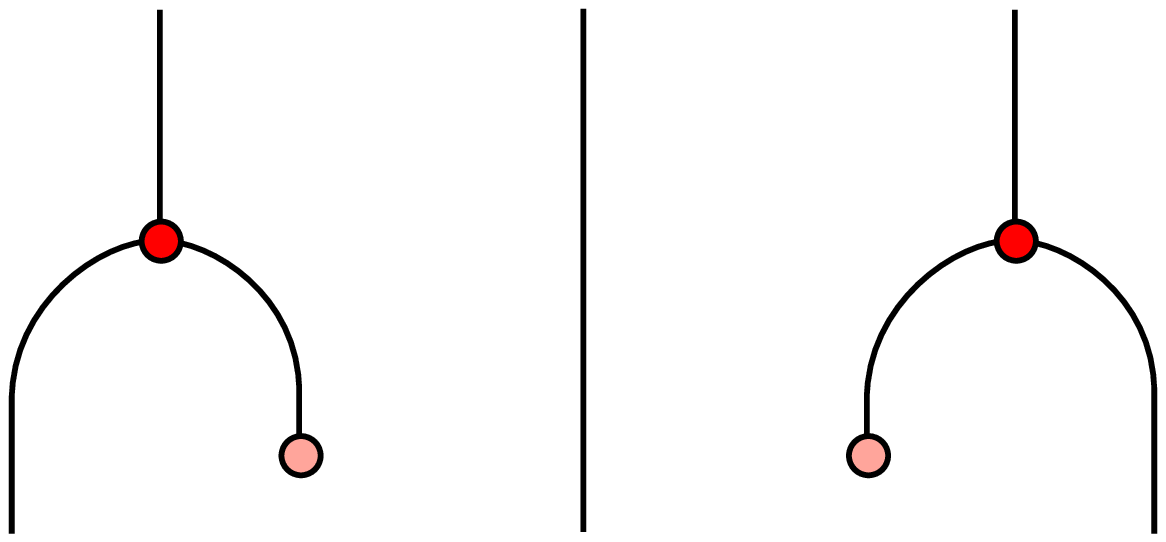}}
		\put(60,0)	{\includegraphics[scale=0.25]{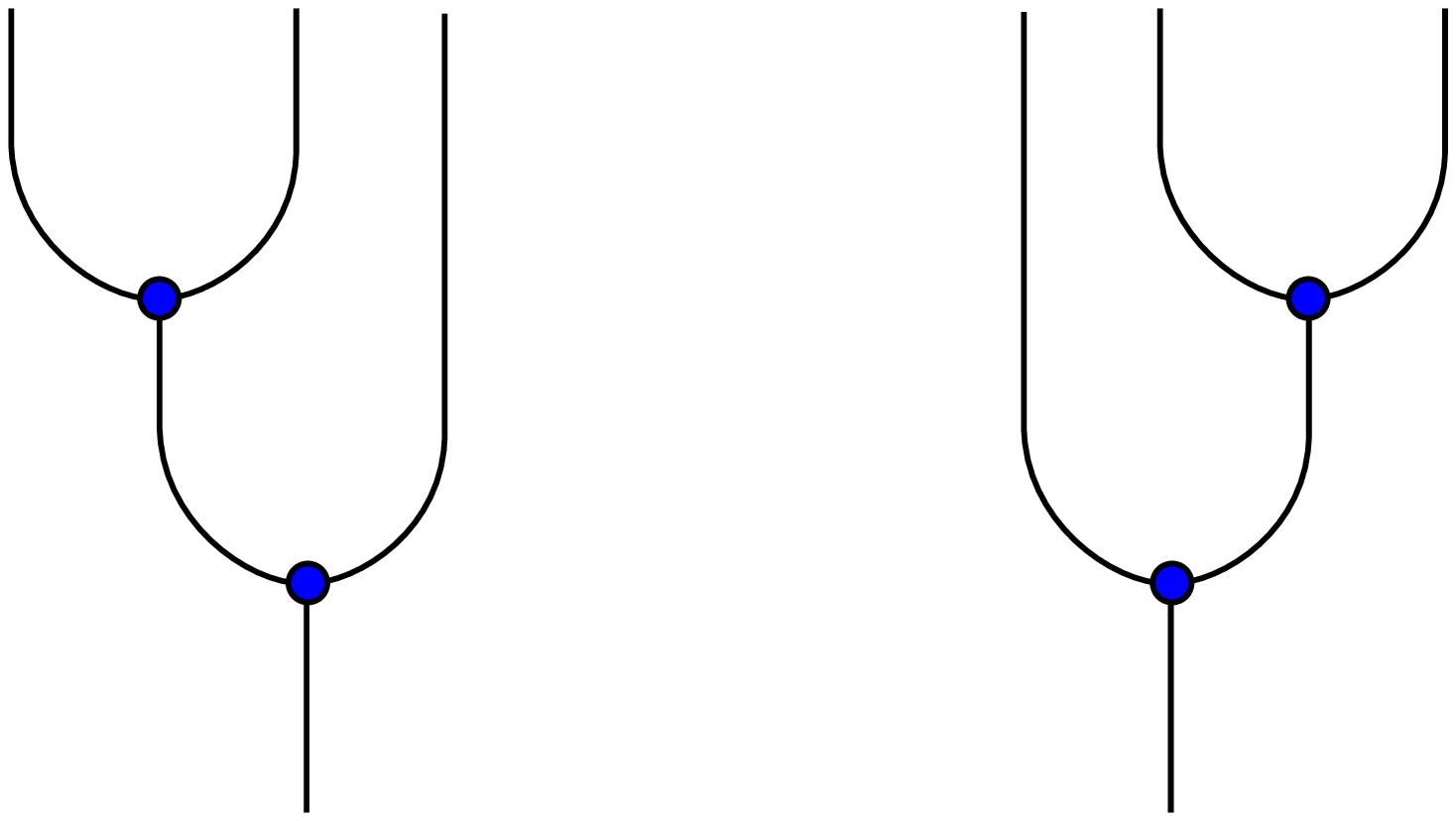}}
		\put(250,0)	{\includegraphics[scale=0.25]{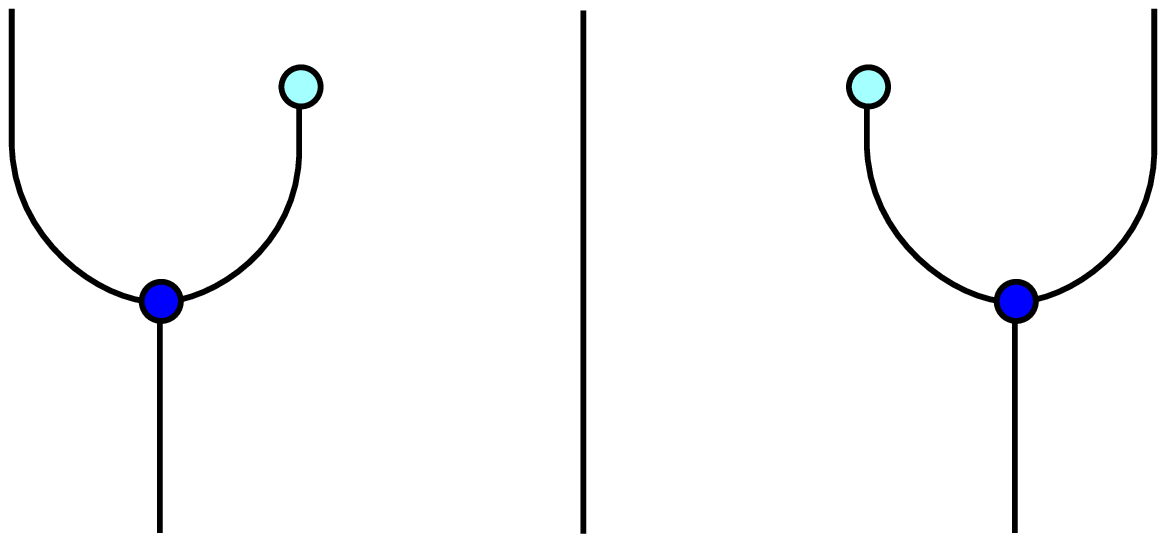}}
	\end{picture}
	\caption{ The graphical notation for a unital algebra and counital coalgebra, with associativity, coassociativity, left- and right-unit and -counit constraints.}
	\label{fig:algcoalg}
\end{figure}

A \underline{Frobenius} algebra in \cC{} is a $5$-tuple $(A,m,\eta,\Delta,\varepsilon)$ s.t. $(A,m,\eta)$ is an algebra, $(A,\Delta,\varepsilon)$ is a coalgebra, and there is the following compatibility condition between the algebra and coalgebra structures:
\be{Frob}
	(\id_A\oti m)\circ(\Delta\oti\id_A)=\Delta\circ m=(m\oti \id_A)\circ(\id_A\oti\Delta),
\ee
which is shown graphically in figure~\ref{fig:frob}.
\begin{figure}[h]
	\centering
	\begin{picture}(200,90)
	\put(0,0)			{\includegraphics[scale=0.3]{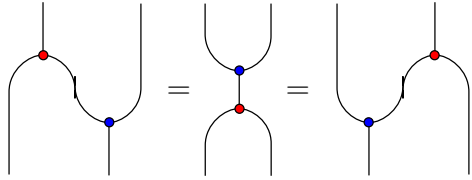}}
	\put(60,30)		{$=$}
	\put(105,30)		{$=$}
	\end{picture}
	\caption{The Frobenius condition.}
	\label{fig:frob}
\end{figure}

An algebra with counit, $(A,m,\eta,\varepsilon)$ is called \underline{symmetric} if the two obvious morphisms from $A$ to $A^\vee$ are equal:
\be{Symm}
	\Phi_1\equiv[(\varepsilon\circ m)\oti \id_{A^\vee}]\circ(\id_A\oti b_A) = [\id_{A^\vee}\oti(\varepsilon\circ m)]\circ (\tilde{b}_A\oti\id_A)\equiv\Phi_2,
\ee
or in graphical notation as shown in figure~\ref{fig:sym}
\begin{figure}[h]
	\centering
	\begin{picture}(160,90)
	\put(0,0)			{\includegraphics[scale=0.3]{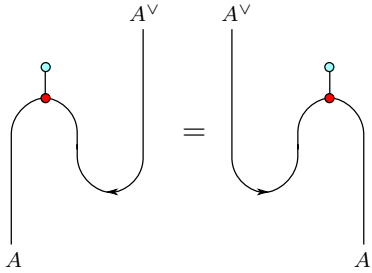}}
	\put(65,40)		{$=$}
	\put(-2,-8)			{\scriptsize $A$}
	\put(45,85)		{\scriptsize $A^\vee$}
	\put(80,85)		{\scriptsize $A^\vee$}
	\put(130,-8)		{\scriptsize $A$}
	\end{picture}
	\caption{The symmetry condition.}
	\label{fig:sym}
\end{figure}

An algebra with counit $(A,m,\eta,\varepsilon)$ is called \underline{non-degenerate} if $\Phi_1$ (or equivalenty $\Phi_2$) as defined above is an isomorphism.

A Frobenius algebra is called normalised \underline{special} if
\be{Spec}
	\varepsilon\circ\eta=\dim(A)\id_\one,\ m\circ\Delta=\id_A.
\ee
We drop the prefix normalised and call a Frobenius algebra special if it satisfies these two conditions.
Specialness is shown in graphical notation in figure~\ref{fig:special}.
\begin{figure}[h]
	\centering
	\begin{picture}(200,100)
	\put(0,0)			{\includegraphics[scale=0.4]{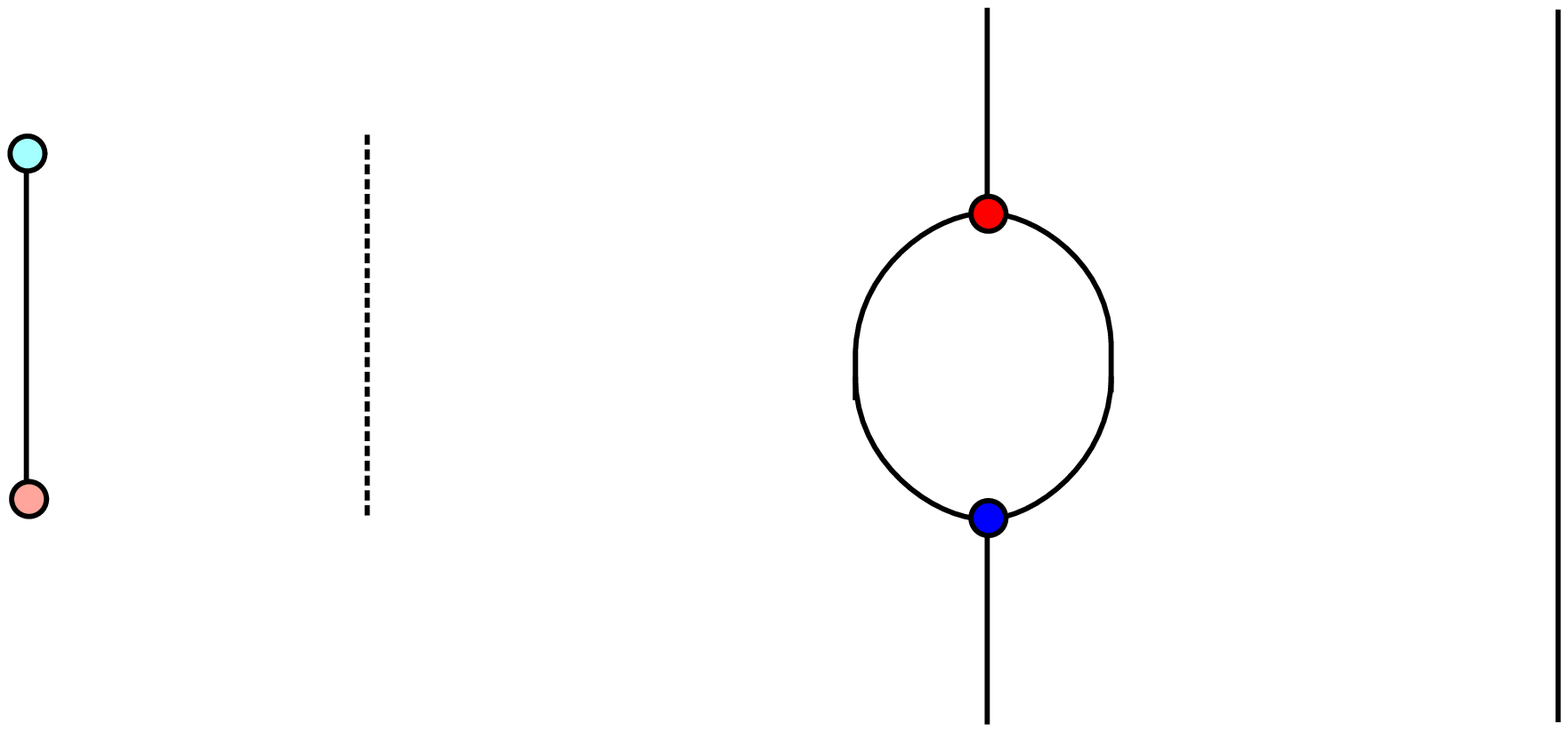}}
	\put(17,50)		{$=d_A$}
	\put(165,50)		{$=$}
	\end{picture}
	\caption{The (normalised) specialness condition.}
	\label{fig:special}
\end{figure}

\noindent Note that there is always the structure of a symmetric special Frobenius algebra on the object $\one$.\\

The notion of (left and right) modules and bimodules of algebras are defined in the obvious way, and we call an algebra $A$ \underline{simple} if it is simple as a bimodule over itself: $\dim_\C\Hom_{A|A}(A,A)=1$. An algebra $A$ is called \underline{haploid} if it contains one copy of the tensor unit: $\dim_\C\Hom(\one,A)=1$.\\[1ex]

\noindent For algebras of the form
\be{algdec}
	A\cong\oplus_{i\in J\subset I} U_i
\ee
(implying haploidity of $A$) we choose retracts $(U_i, e_{U_i\prec A},r_{U_i\prec A})$ corresponding to every simple subobject. These distinguished retracts are drawn as follows

\begin{figure}[h]
	\centering
	\begin{picture}(150,80)(0,0)
		\put(0,0)		{\includegraphics[scale=0.4]{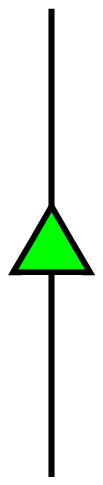}}
		\put(1,-8)		{\scriptsize $U_i$}
		\put(1,62)		{\scriptsize $A$}
		\put(10,28)	{$e_{U_i\prec A}$}
		\put(120,0)	{\includegraphics[scale=0.4]{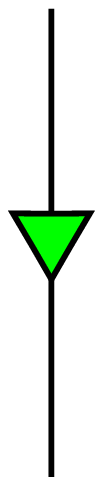}}
		\put(121,-8)	{\scriptsize $A$}
		\put(121,62)	{\scriptsize $U_i$}
		\put(130,20)	{$r_{U_i\prec A}$}
	\end{picture}
	\caption{Graphical notation for the distingished retracts of an algebra $A$.}
\end{figure}

In particular we let $e_{\one\prec A}=\eta$, so if $A$ is (normalised) special we have $r_{\one\prec A}=\dim(A)^{-1}\varepsilon$.
Using these retracts we define the components of the multiplication and comultiplication as shown in figure~\ref{fig:multcomp}.
\begin{figure}[h]
	\centering
	\begin{picture}(150,80)(0,0)
		\put(0,0)		{\includegraphics[scale=0.4]{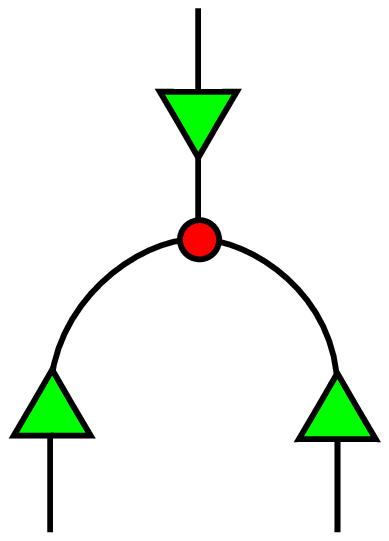}}
		\put(-55,30)	{$m_{ij}^{\phantom{ij}k} \lambda_{(ij)k} =$}
		\put(0,-8)		{\scriptsize $U_i$}
		\put(37,-8)	{\scriptsize $U_j$}
		\put(17,65)	{\scriptsize $U_k$}
		\put(120,0)	{\includegraphics[scale=0.4]{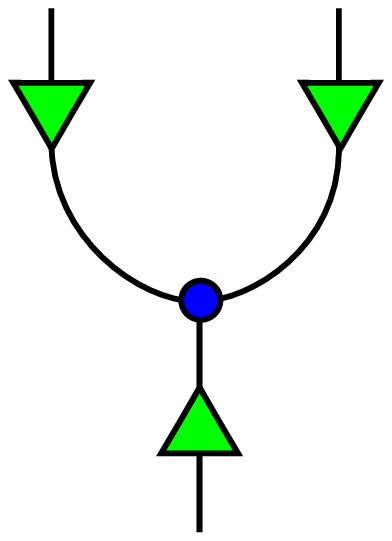}}
		\put(65,30)	{$\Delta_i^{\phantom{i}jk}\bar{\lambda}^{(jk)i}=$}
		\put(120,65)		{\scriptsize $U_j$}
		\put(157,65)	{\scriptsize $U_k$}
		\put(139,-9)	{\scriptsize $U_i$}
	\end{picture}
	\caption{Definition of the components of the multiplication and comultiplication in terms of the distinguished retracts.}
	\label{fig:multcomp}
\end{figure}

We will need the following result from~\cite{FRS1}
\begin{lemma}(Lemma 3.7 b) \cite{FRS1})
Let  $(A,m,\eta)$ be an algebra and let $\varepsilon\in\Hom(A,\one)$ be a morphism s.t. $(A,m,\eta,\varepsilon)$ is a non-degenerate algebra,
 then there is a unique structure of symmetric special Frobenius algebra on $A$.
\end{lemma}

As a consequence, to check whether a triple $(A,m,\eta)$ can be extended to a ssFa we only need to check associativity, unitality and find a $\varepsilon$ s.t. the resulting structure is non-degenerate. If $A$ is haploid we can already choose $\varepsilon=d_A r_{\one\prec A}$.
If we restrict to algebras of the form above, i.e. where $\dim_\C\Hom(U_i,A)\in\{0,1\}$ for any simple object $U_i$, the restrictions formed by demanding unitality, associativity and non-degeneracy can be expressed rather compactly in terms of the components of $m$. In particular we have
\begin{itemize}
\item
Unitality:
\be{unitcomp}
	m_{0i}^{\phantom{oi}i}=1=m_{i0}^{\phantom{i0}i}\ \forall U_i\prec A
\ee
\item
Non-degeneracy:
\be{nondegcomp}
	m_{i\bar{\iota}}^{\phantom{i\bar{\iota}}0}\neq 0\ \forall U_i\prec A
\ee
\end{itemize}

\paragraph{Endofunctors Related to Algebras}~\\
Let $A$ be a symmetric special Frobenius algebra, and consider  the endomorphisms $P^{l/r}_A$ shown in figure below. It follows straightforwardly from the properties of $A$ that $P^{l/r}_A$ are idempotents, see lemma 5.2 of \cite{FRS1}.

\begin{figure}[h]
	\centering
	\begin{picture}(200,100)(0,0)
		\put(0,40)	{$P^l_A=$}
		\put(30,0)		{\includegraphics[scale=0.4]{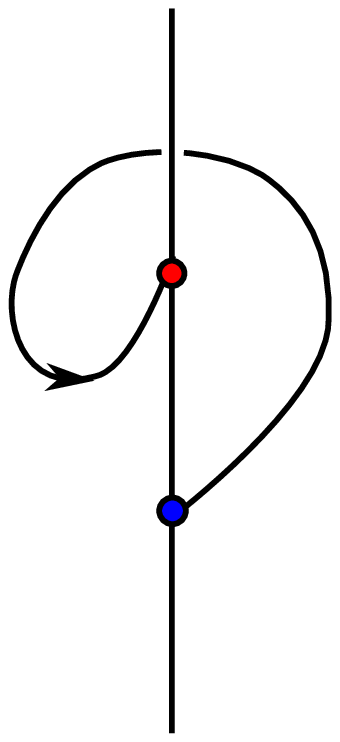}}
		\put(90,40)	{$P^r_A=$}
		\put(120,0)	{\includegraphics[scale=0.4]{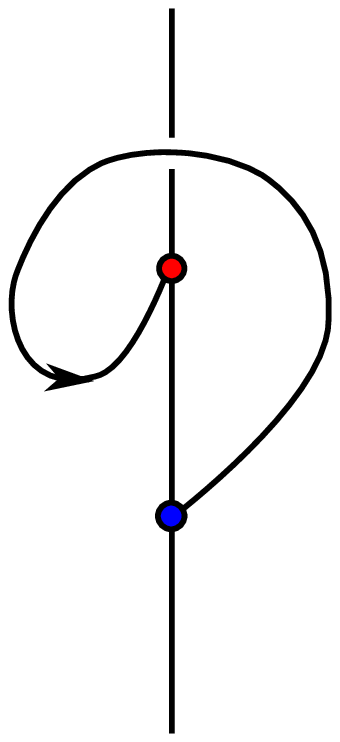}}
	\end{picture}
\end{figure}

\begin{definition}
The images $(C_{l/r}(A),e_{l/r},r_{l/r})$ of $P_A^{l/r}$ are called the {\em left center} respectively the {\em right center} of $A$.
\end{definition}
The left and right centers have the following properties
\begin{itemize}
	\item $C\equiv C_{l/r}(A)$ has the structure of a commutative symmetric special Frobenius algebra, with $m_C=r\circ m\circ(e\oti e)$, $,\eta_C=r\circ\eta$, $\Delta_C=\zeta(r\oti r)\circ\Delta\circ e$, $\varepsilon_C=\zeta\varepsilon\circ e$ for $\zeta=d_Cd_A^{-1}$. If $A$ is simple, so is $C$. (proposition 2.37, \cite{FrFRS})
	\item $C$ has trivial twist, $\theta_C=\id_C$ (lemma 2.33, \cite{FrFRS})
	\item In CFT, $C$ has the interpretation as the maximal extension of the chiral algebra w.r.t a given modular invariant torus partition function
\end{itemize}
For each object $U$ we define two endomorphisms, $P^l_A(U)$ and $P^r_A(U)$, of $A\oti U$ as in figure~\ref{fig:endofunctor}.

\begin{figure}[h]
	\centering
	\begin{picture}(200,100)(0,0)
		\put(25,0)		{\includegraphics[scale=0.4]{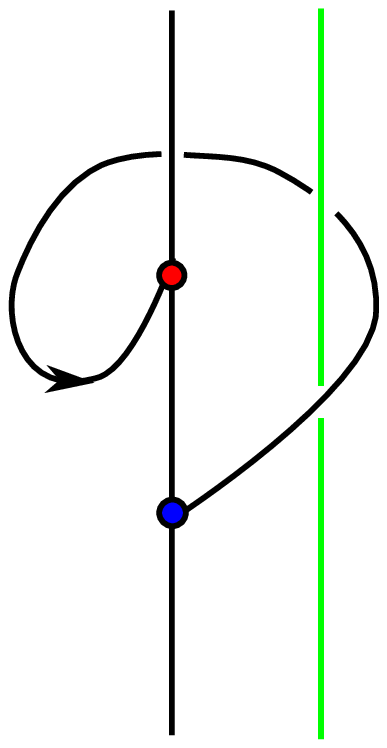}}
		\put(-30,40)		{$P^l_A(U)=$}
		\put(40,-8)	{\scriptsize $A$}
		\put(58,-8)	{\scriptsize $U$}
		\put(40,88)	{\scriptsize $A$}
		\put(58,88)		{\scriptsize $U$}
		\put(140,0)	{\includegraphics[scale=0.4]{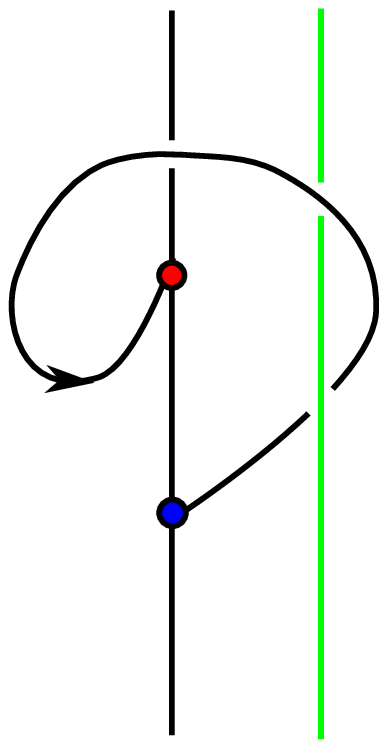}}
		\put(90,40)	{$P^r_A(U)=$}
		\put(155,-8)	{\scriptsize $A$}
		\put(173,-8)	{\scriptsize $U$}
		\put(155,88)	{\scriptsize $A$}
		\put(173,88)	{\scriptsize $U$}
	\end{picture}
	\caption{Idempotent endomorphisms defining the two endofunctors $E^{l/r}_A$.}
	\label{fig:endofunctor}
\end{figure}

Similarly as for $P^{l/r}_A$ the properties of $A$ imply that $P^{l/r}_A(U)$ are idempotents, see lemma 5.2 of \cite{FRS1}.

\begin{definition} The two endofunctors $E^{l/r}_A\in\End(\cC)$ are defined through
	\begin{itemize}
		\item $U\in\Obj(\cC)\mapsto E^{l/r}_A(U)\equiv\im(P^{l/r}_A(U))$
		\item $f\in\Mor(\cC)\mapsto E^{l/r}_A(f)\equiv r_{l/r}\circ f\circ e_{l/r}$
	\end{itemize}
\end{definition}
Note that by taking $U=\one$ we get $C_{l/r}(A)=E^{l/r}_A(\one)$.
\begin{remark} The functors $E^{l/r}_A$ are not tensor functors since, in general, $E^{l/r}_A(U)\oti E^{l/r}_A(V)\ncong E^{l/r}_A(U\oti V)$. Using these it is, however, possible to construct tensor functors. The objects $E^{l/r}_A(U)$ carry a natural structure of modules of $C_{l/r}(A)$, and they furthermore have the property of trivialising the braiding w.r.t. $C_{l/r}(A)$. Such modules were called {\em local} modules in \cite{FrFRS}. The category $\cC^{loc}_{C_{l/r}(A)}$ of local modules of the left or right center of a symmetric special Frobenius algebra $A$ happens to be a ribbon category, and if \cC{} is modular then so is $\cC^{loc}_{C_{l/r}(A)}$. By interpreting $E^{l/r}_A$ as functors from \cC{} to $\cC^{loc}_{C_{l/r}(A)}$, they are both tensor functors (see \cite{FrFRS}).
\end{remark}

Denote by $\{\chi_i\}_{i\in I}$ the basis of $\bl(T^2)$ given by the coevaluation morphisms $b_{U_i}$, i.e. $\chi_i$ is given by the invariant of a solid torus with a single ribbon of trivial framing labelled by $U_i$ traversing the non-contractible cycle once, and let $\{\bar{\chi}_i\}_{i\in I}$ be the dual basis. A crucial result is given by the following
\begin{proposition}
The modular invariant torus partition function $Z(A)=\sum_{i,j\in I}Z_{ij}(A)\chi_i\oti\bar{\chi}_j$ corresponding to a symmetric special Frobenius algebra $A$, of a rational CFT based on the modular tensor category \cC{}, is given by
	\be{modinv}
		Z_{ij}(A)=\dim_\C\Hom(E^l_A(U_i),U_j)
	\ee
\end{proposition}

\noindent Proof: This follows immediately from eq. $(2.49)$ and proposition $3.6$ of \cite{FrFRS}.$\Box$\\[1ex]

\noindent We will often drop the reference to $A$ and simply write $Z_{ij}$ for the coefficients.
\begin{remark}
\label{rem:morita}
\begin{itemize}
	\item It follows from the discussion in section 3.3 of \cite{FrFRSduality} that if $A$ and $B$ are two Morita equivalent symmetric special Frobenius algebras in \cC, then $Z_{ij}(A)=Z_{ij}(B)$. Combining the same discussion with the uniqueness result of \cite{FjFRS2} implies that two rational CFT's based algebras $A$ and $B$ in the same modular category \cC{} are equivalent (see \cite{FjFRS2} for the definition of equivalence) if and only if $A$ and $B$ are Morita equivalent. It is, however, not excluded at this point that there may exist two non-Morita equivalent algebras with the same torus partition function.
	\item In the CFT litterature, including \cite{FRS1,FrFRS}, it is conventional to express the torus partition functions in a different basis, namely $\{\chi_i\oti\tilde{\chi}_j\}_{i,j\in I}$, where $\tilde{\chi}_j\equiv\bar{\chi}_{\bar{j}}$. Thus what here appears as the identity matrix is in the CFT litterature the matrix $C_{ij}\equiv\delta_{i,\bar{j}}$, the {\em charge conjugation} matrix, and the partition function is referred to as the charge conjugation modular invariant. For our purposes it is convenient to identify the partition function with an endomorphism of the corresponding space of conformal blocks, and it is therefore more convenient to have the identity endomorphism being represented by the identity matrix.
\end{itemize}
\end{remark}

Finally, the following result will be useful.

\begin{proposition}
	Let $A$ be a symmetric special Frobenius algebra in a modular tensor category such that $Z_{ij}(A){\propto\!\!\!\!\! \slash\ }  \delta_{ij}$. Then there is a subobject $B\prec A$ with the structure of a symmetric special Frobenius algebra, and such that $B$ is Morita equivalent to a haploid symmetric special Frobenius algebra $C$ with $Z_{ij}(A){\propto\!\!\!\!\! \slash\ }  \delta_{ij}$.\end{proposition}

\begin{proof}
	Lemma 5.23 of~\cite{FuS} implies that $A$ is semisimple in the sense of~\cite{Ostrik}. If $A$ is indecomposable, then the corrollary to Theorem~$1$ in~\cite{Ostrik} implies that $A$ is Morita equivalent to a haploid non-degenerate algebra, and thus to a haploid symmetric special Frobenius algebra $B$, and $Z_{ij}(B)){\propto\!\!\!\!\! \slash\ }  \delta_{ij}$ according to remark~\ref{rem:morita}. If $A$ is decomposable, then there is a sub object $A_1\prec A$ with the structure of a non-degenerate algebra such that $Z_{ij}(A_1) {\propto\!\!\!\!\! \slash\ }  \delta_{ij}$ since $Z(A\oplus B)=Z(A)+Z(B)$, and we apply the same argument to $A_1$. $A$ has only a finite number of subobjects, so the procedure will end in a haploid symmetric special Frobenius algebra $C$ such that $Z_{ij}(C){\propto\!\!\!\!\! \slash\ }  \delta_{ij}$.
\end{proof}

\paragraph{Algebras in $\cC_k$}~\\

The modular invariant torus partition functions of the $su(2)$ WZW model has a well-known ADE classification~\cite{CaItZu}. The modular invariant of A type is the diagonal one $Z_{ij}=\delta_{ij}$, which exists for all values of $k$. For odd $k$, this is also the only one, and the methods described here cannot be used to analyze mapping class group representations with respect to reducibility.
For all even $k$ there is another, non-diagonal, modular invariant corresponding to the D in ADE, which we will use here. Finally there are three exceptional cases: for $k=10, 16$ and $28$ there are (non-diagonal) modular invariants corresponding to the $E_6$, $E_7$ and $E_8$ Dynkin diagrams.
Non-degenerate algebras corresponding to $D_{2n}$, $E_6$ and $E_8$ were constructed in \cite{KO}, and corresponding to $D_{2n+1}$ and $E_7$ in \cite{Ostrik}. Corresponding structures in the language of nets of subfactors on the circle were, however, constructed earlier in \cite{BE}, and it is a relatively straightforward task to translate between the two languages. In particular, the object underlying an algebra is obtained immediately.

\begin{table}[h]
\centering
	\begin{tabular}{c|c|c|c}
	type & $Z_{ij}$ & levels & algebra object\\
	\hline
	$A$ & $\delta_{ij}$ & $k\in\Z_+$ & $\one$\\
	$D$ & non-diagonal & $k\in 2\Z_+$, $k\geq 4$ & $\one\oplus U_k$\\
	$E_6$ & non-diagonal & $k=10$ & $\one\oplus U_6$\\
	$E_7$ & non-diagonal & $k=16$ & $\one\oplus U_8\oplus U_{16}$\\
	$E_8$ & non-diagonal & $k=28$ & $\one\oplus U_{10}\oplus U_{18}\oplus U_{28}$
	\end{tabular}
	\caption{The ADE classification of modular invariant torus partition functions for the $\mathfrak{su}(2)$ WZW model. It is indicated at what levels the different types occur. For each type the object underlying a simple algebra in the corresponding Morita class is given. The algebra structure is unique on all but the $E_8$ one, which nevertheless has a unique structure of a commutative algebra.}
	\label{tab:ADEclass}
\end{table}

\section{Reducibility of mapping class group representations}\label{sec:reducibility}
We spend the first section proving theorem \ref{thm:reducibility}, and in the second section we investigate in more detail the consequences of this theorem in the $\mathfrak{su}(2)$ case, concluding with a proof of theorem \ref{thm:su2}.
Fix notation as follows:
\begin{itemize}
	\item Let $\extS_g$ denote an extended surface of genus $g$ with no marked points
	\item For a given modular tensor category \cC{}, we denote by $\bl(\extS)$ the vector space corresponding to the extended surface $\extS$ and by $Z(\eCob)$ the invariant corresponding to the extended cobordism $\eCob$, given by the TQFT associated to \cC.
\end{itemize}

\subsection{Proof of theorem \ref{thm:reducibility}}\label{sec:redfromalg}
The strategy of the proof is to use Schur's lemma. Given a special symmetric Frobenius algebra
 with the required property, we explicitly construct an element $P_g\in\End(\bl(\extS_g))$ in the commutant of the representation of the mapping class group, and show that $P_g{\propto\!\!\!\!\! \slash\ } \id$.
The proof of theorem~\ref{thm:reducibility} requires some preparation, and
we begin by defining extended cobordisms whose invariants provide the elements $P_g$ in the commutant.\\

\noindent Let $T$ be a (Poincar\'{e}) dual triangulation of a surface $\extS$, i.e. an embedding of a trivalent graph in $\extS$ such that the Poincar\'{e} dual graph is a triangulation of $\extS$, and let $A$ be a symmetric special Frobenius algebra in a ribbon category \cC. By {\bf labelling T with $A$} we mean the following: Embed a (oriented and core-oriented) ribbon graph in $\extS$ such that each edge between two vertices of $T$ is covered by a ribbon labeled by $A$, and each trivalent vertex connects three edges by {\em one choice} of planar graph with coupons labelled by combinations of $m$, $\Delta$, $\varepsilon\circ m$, or $\Delta\circ\eta$, in such a way that any region of the graph is consistently interpreted as a morphism in \cC{} with the convention of letting core-orientation being the ''upwards'' direction.\\
There are thus many consistent ways of labelling $T$ with $A$, but it is not difficult to convince oneself that the properties {\em symmetry}, {\em specialness} and {\em Frobenius}
 are enough to guarantee that all possible consistent labellings are equivalent, as morphisms in \cC. An actual proof can be found by combining results from \cite{FRS1, FRS2, FjFRS1}.
By equipping each edge of $T$ with an orientation such that the edges at any vertex are either two in-going and one out-going, or one in-going and two out-going, there is an unambiguous way to label $T$ by $A$, where every trivalent vertex is covered with a coupon labelled either by $m$ or by $\Delta$. In the following we will indicate a particular labelling in this way.

\begin{definition}
Let $A$ be a symmetric, special Frobenius algebra in \cC, and let \extS$_g$ be an extended surface of genus $g$ with no marked points. Pick a dual triangulation $T$ of \extS$_g$, and consider the three-manifold $\eCob[\extS_g]\equiv \extS\times[-1,1]$ equipped with the natural orientation obtained from the orientation of \extS$_g$ and the interval. Define an extended cobordism $\eCob[\extS_g,A,T]:\extS_g\rightarrow\extS_g$ by labelling $T$ on $\extS_g\times\{0\}\subset\eCob[\extS_g]$ with $A$.
We define an element $P[\extS_g,A,T]$ of $\End(\bl(\extS_g))$ by
\be{Pdef}
P[\extS_g,A,T]\equiv Z(M[\extS_g,A,T]).
\ee
\end{definition}

\begin{proposition}(Proposition {\rm 3.2}, \cite{FjFRS1})
$P[\extS_g,A,T]$ is independent of the choice of dual triangulation $T$.
\end{proposition}

\begin{proposition} $P[\extS_g,A,T]$ commutes with the action of the mapping class group: If $\rho$ denotes the representation on $\bl(\extS_g)$, we have $$P[\extS_g,A,T]\circ\rho(f)=\rho(f)\circ P[\extS_g,A,T]$$ for any mapping class $f$.
\end{proposition}
\begin{proof}
Follows immediately from Theorem~{\rm 2.2} of \cite{FjFRS1}.
\end{proof}

~\\
\noindent Since $P[\extS_g,A,T]$ is independent of triangulation we will drop the $T$. Furthermore, when it is clear from the context which algebra and surface is being used, we will abuse notation and simply write $P_g\equiv P[\extS_g,A]$.
For the remaining discussion we will choose a directed dual triangulation for each genus $g$ extended surface, as shown in figure~\ref{fig:genusgT}. As indicated, an orientation of the edges has been chosen, allowing an unambiguous labeling of the coupon covering any given vertex.
\begin{figure}[h]
\centering
	\includegraphics[scale=0.3]{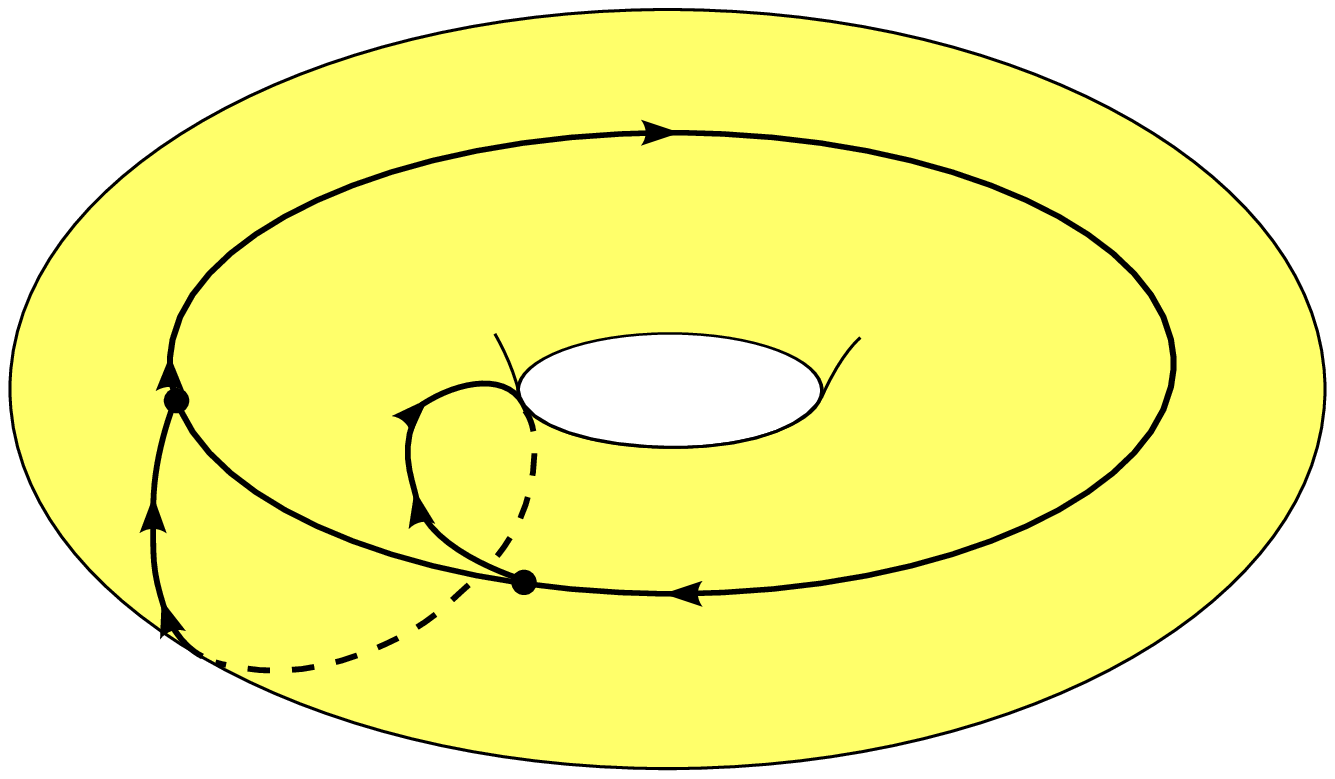}
	\includegraphics[scale=0.5]{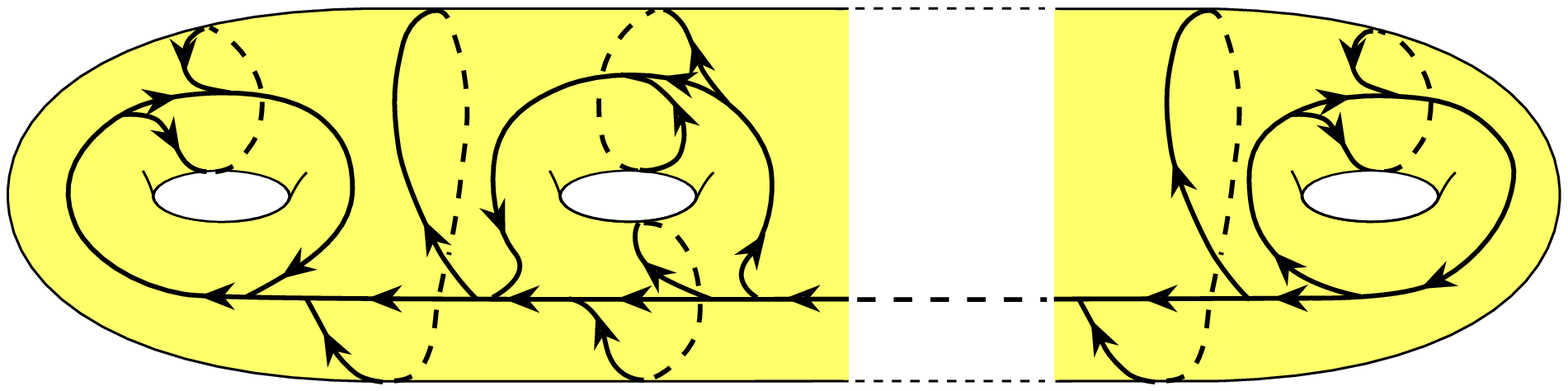}
	\caption{Our fix choice of dual triangulation of an extended surface of genus $g=1$, resp. $g\geq 2$}
	\label{fig:genusgT}
\end{figure}
Next we will see how $P_g$ acts on a standard basis element.
The standard basis of $\bl(\extS_g)$ is a tuple of (non-zero) $\lambda$'s given by a pants decomposition of \extS$_g$ by viewing $\extS_g$ as the boundary of a handle body $\mathrm{H}_g$ with an embedded ribbon graph, reduced by the corresponding pants decomposition of $\mathrm{H}_g$ to a tree as illustrated in figure~\ref{fig:standardbasis}.
\begin{figure}[h]
	\centering
	\includegraphics[scale=0.4, angle=90]{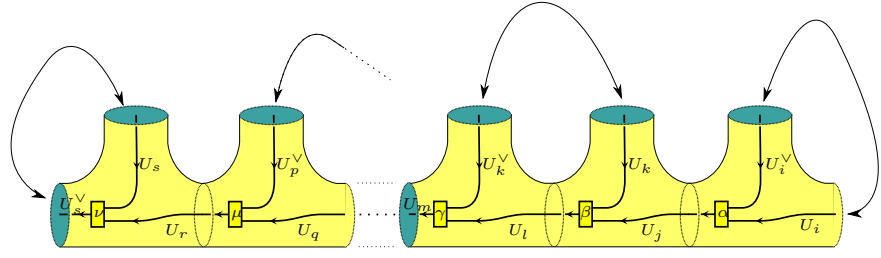}
	\put(-30,7)		{\tiny $U_i$}
	\put(-90,5)		{\tiny $U_j$}
	\put(-140,5)		{\tiny $U_l$}
	\put(-43,30)		{\tiny $U_i^\vee$}
	\put(-95,30)		{\tiny $U_k$}
	\put(-150,30)		{\tiny $U_k^\vee$}
	\put(-180,15)		{\tiny $U_m$}
	\put(-220,5)		{\tiny $U_q$}
	\put(-270,5)		{\tiny $U_r$}
	\put(-228,30)		{\tiny $U_p^\vee$}
	\put(-280,30)		{\tiny $U_s$}
	\put(-310,15)		{\tiny $U_s^\vee$}
	\put(-61,12)		{\tiny $\alpha$}
	\put(-113,12)		{\tiny $\beta$}
	\put(-168,12)		{\tiny $\gamma$}
	\put(-245,12)		{\tiny $\mu$}
	\put(-297,12)		{\tiny $\nu$}
	\caption{A tree graph inside a pants-decomposed handle body corresponding to the basis element $(\lambda_{(i\overline{\iota})j}^\alpha,\lambda_{(jk)l}^\beta,\lambda_{(l\overline{k})m}^\gamma,\ldots,\lambda_{(q\overline{p})r}^\mu,\lambda_{(rs)\overline{s}}^\nu)$.}
	\label{fig:standardbasis}
\end{figure}

The action of $P_g$ on a standard basis element is given by glueing the corresponding cobordisms and taking the invariant.~\footnote{It is obvious that the corresponding Maslov index vanishes in this simple case.} In the present case, the resulting cobordism is again a handlebody, but with the following ribbon graph inserted.
\begin{figure}[h]
	\centering
	\includegraphics[scale=0.5, angle=90]{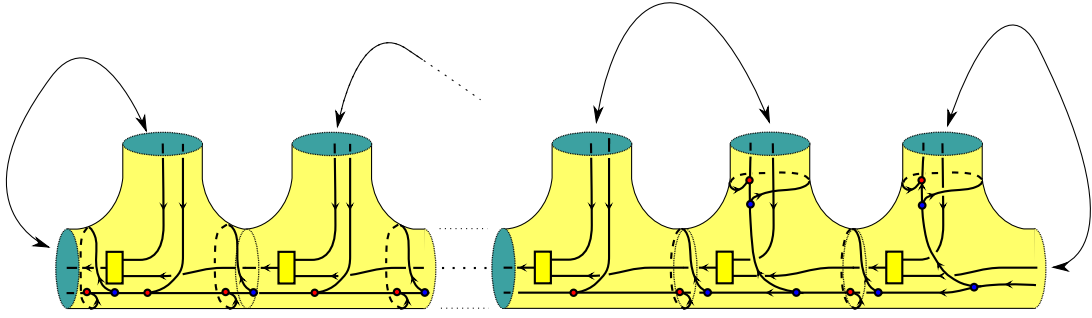}
	\caption{The extended cobordism resulting from glueing $\eCob[\extS_g,A,T]:\extS_g\rightarrow\extS_g$  onto the handlebody shown in figure~\ref{fig:standardbasis}. Labels are supressed.}
	\label{fig:PonH}
\end{figure}

Note that the ribbon graph represents a morphism involving a number of projectors $P_A^l(U_i)$. Consider a tubular section of the handle body in figure~\ref{fig:PonH} containing only such a projector. Using the representation $P_A^l(U_i)=e_i\circ r_i$ we conclude

\begin{figure}[h]
	\centering
	\begin{picture}(300,100)(0,0)
		\put(50,0)		{\includegraphics[scale=0.3]{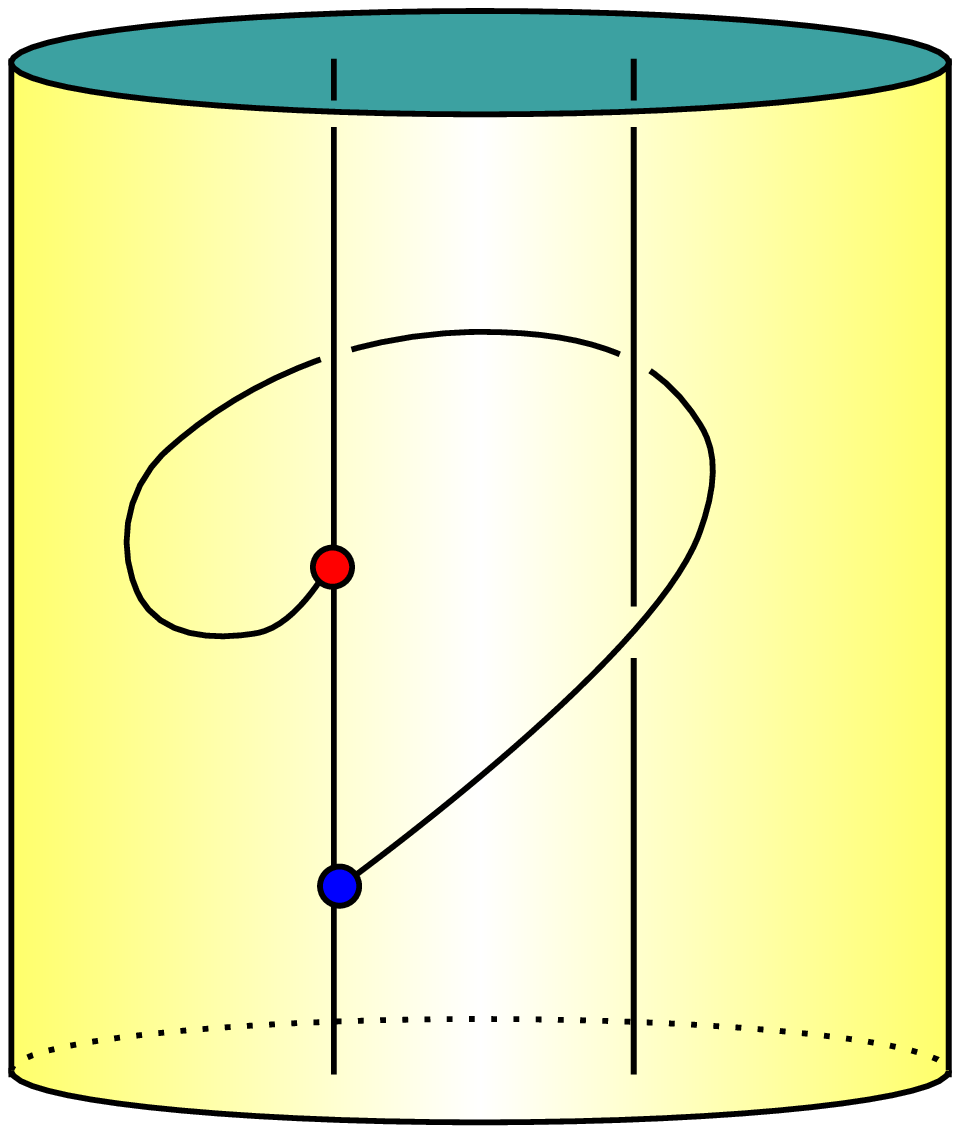}}
		\put(70,30)	{\scriptsize $A$}
		\put(107,30)	{\scriptsize $U_i$}
		\put(160,50)	{$\sim$}
		\put(200,0)	{\includegraphics[scale=0.3]{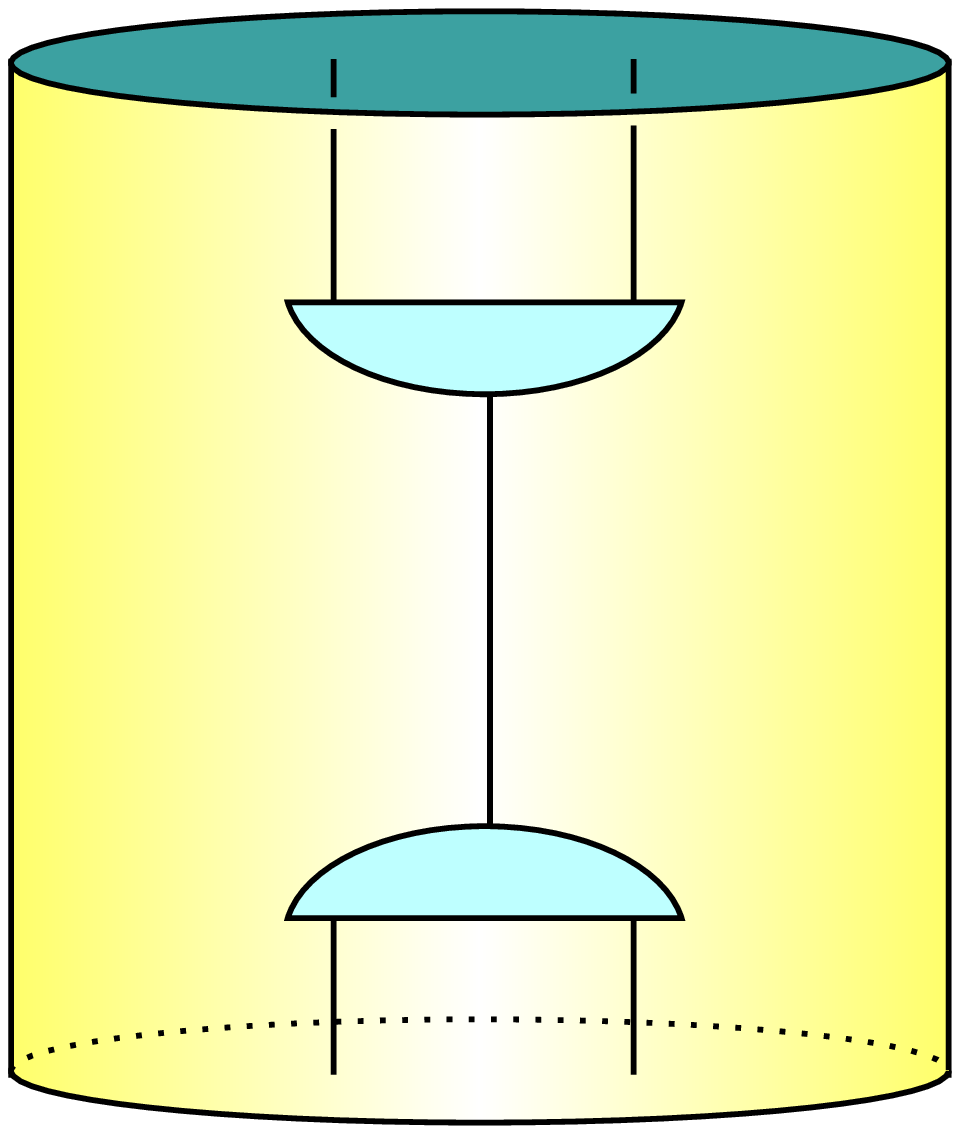}}
		\put(220,10)	{\scriptsize $A$}
		\put(257,10)	{\scriptsize $U_i$}
		\put(220,80)	{\scriptsize $A$}
		\put(257,80)	{\scriptsize $U_i$}
		\put(245,45)	{\scriptsize $E^l_A(U_i)$}
	\end{picture}
	\caption{The projectors $P^l_A(U_i)$ inside the handlebody can be represented by a composition $e_i\circ r_i$.}
	\label{fig:ProjIm}
\end{figure}

where $\sim$ means that the invariants of two extended cobordisms differing only in the depicted region, coincide.
Apply this to every occurence of $P_A^l$ inside the handle body, and slide each of the $e_i$ and $r_j$ morphisms towards the vertices of the ribbon graph. Locally, in a tubular section around a vertex labelled by $\lambda^\alpha_{(ij)k}\in\Hom(U_i\oti U_j,U_k)$, the extended cobordism then takes one of the two forms shown in figure~\ref{fig:Vertex}.

\begin{figure}[h]
	\centering
	\begin{picture}(300,100)(0,0)
		\put(50,0)		{\includegraphics[scale=0.3]{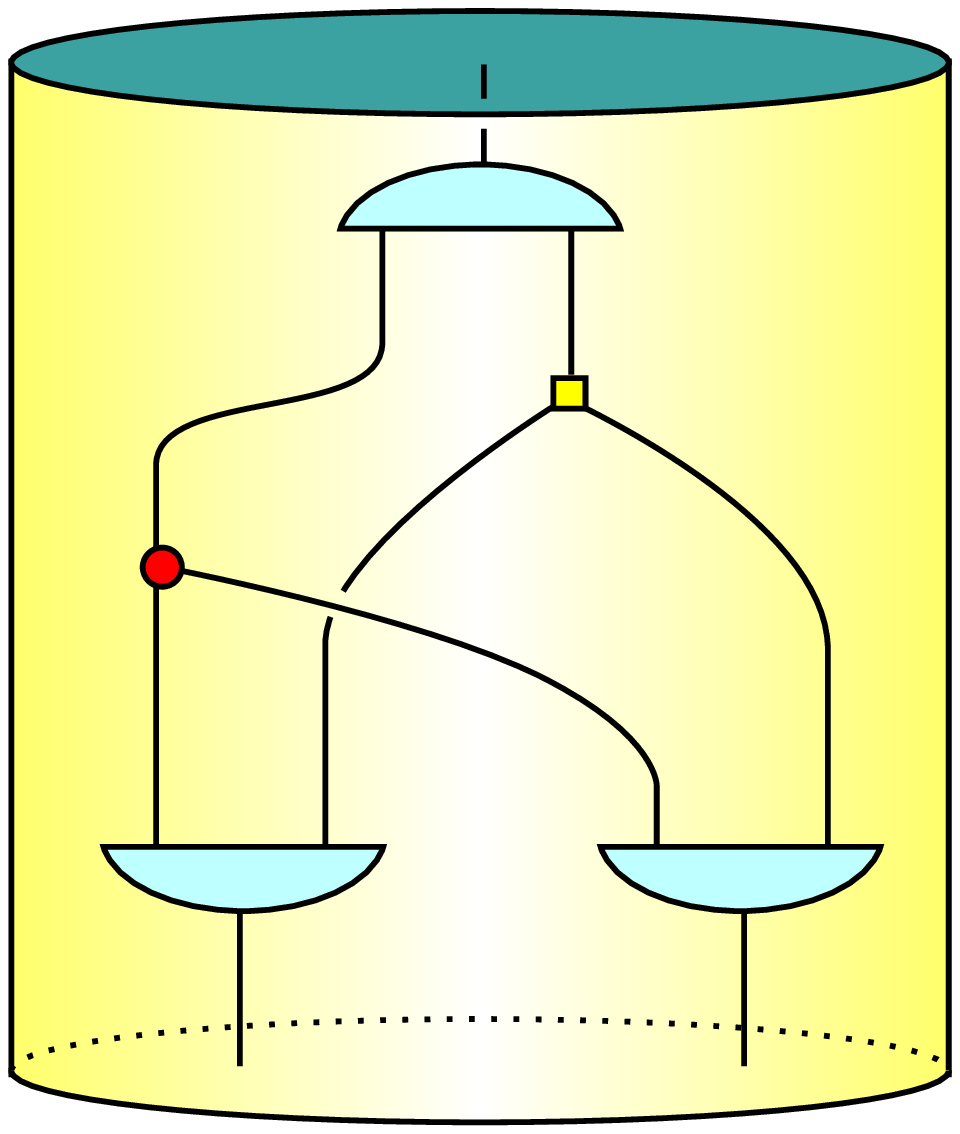}}
		\put(200,0)	{\includegraphics[scale=0.3]{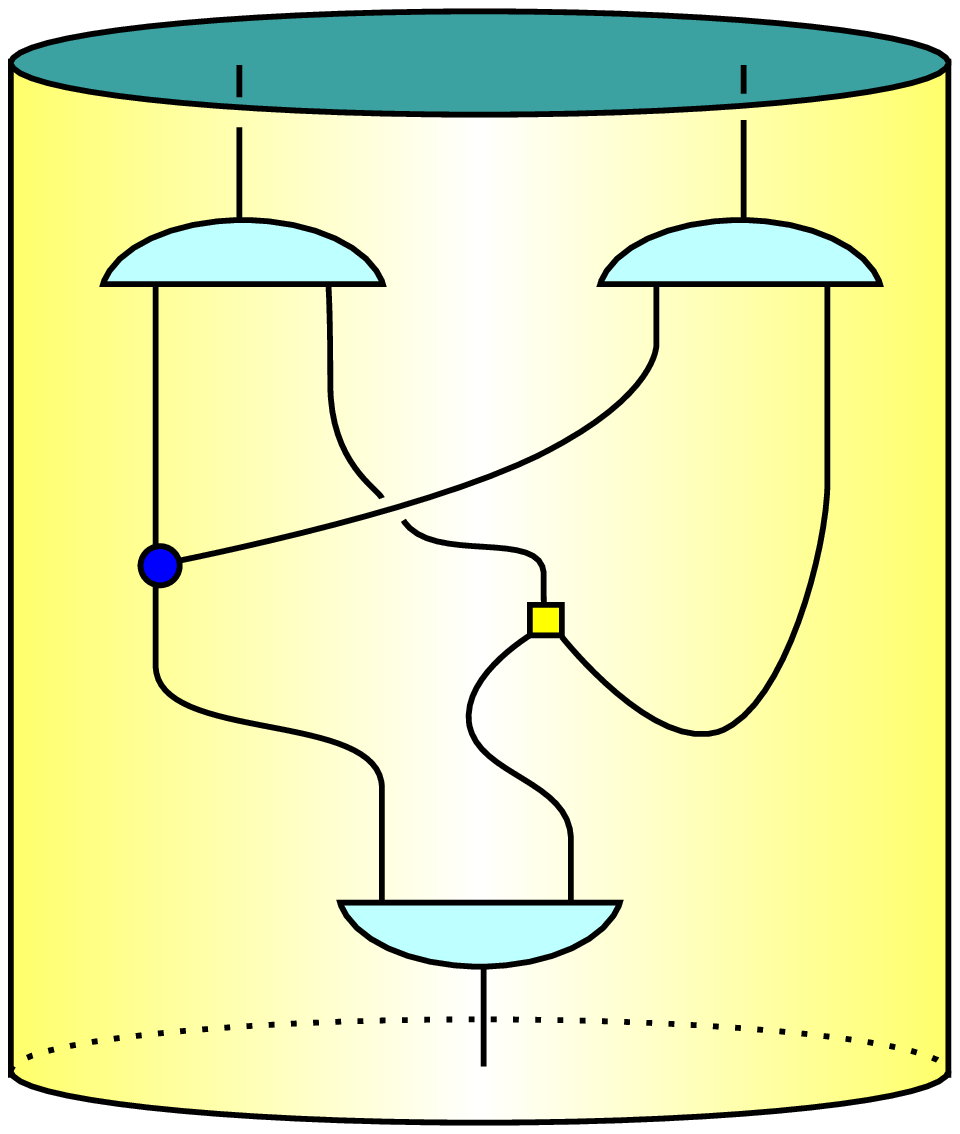}}
	\end{picture}
	\caption{The two possibilities of neighbourhoods around trivalent morphisms.}
	\label{fig:Vertex}
\end{figure}

Since the endofunctor $E^l_A$ is not a tensor functor, it is not straightforward to reduce the corresponding morphisms further.
Note, however, that if all three objects $U_i$, $U_j$, and $U_k$ are the tensor unit, the morphisms correspond to those
defining the multiplication and comultiplication of $C_l(A)$ (up to a factor of $dim(C_l(A))/dim(A)$ for the comultiplication).~\\[1ex]

\noindent For any $g\geq 1$ we define a special element $v^g_U\in\bl(\extS_g)$ for all $U\in \Obj(\cC)$ (if $U$ is simple, abbreviate $v^g_i\equiv v^g_{U_i}$) as follows. Consider the cobordism of the type shown in figure~\ref{fig:standardbasis} with all but one of the objects the tensor unit (s.t. the corresponding ribbons can be completely left out in the handle body), and all morphisms given by unit constraints (since we're considering a strict category, these are all identity morphisms), and take the ''last'' object to be $U$. Denote the resulting cobordism, shown in figure~\ref{fig:basisel}, by $M^g_U$, and define $v^g_U$ to be the corresponding element of the standard basis. Note that $v^1_i=\chi_i$ with the notation used in section~\ref{ssec:algebras}.

\begin{figure}[h]
	\centering
	\includegraphics[scale=0.5, angle=90]{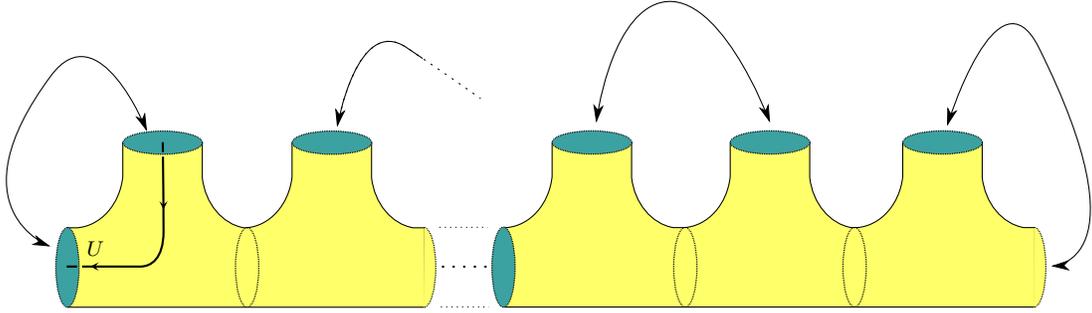}
	\put(-380,20)	{\scriptsize $U$}
	\caption{The cobordism representing the element $v^g_U$.}
	\label{fig:basisel}
\end{figure}

It follows immediately that $P_1(v^1_i)$ is precisely $v^1_{E^l_A(U_i)}$.
When $g\geq 2$, the element $P_g(v_i)$ is represented by a ribbon graph in a handle body where, following the discussion above, all but one of the ribbons are labelled by $C_l(A)$ and all but one of the coupons are labelled by multiplications and comultiplications of $C_l(A)$. The "last" ribbon is labelled by the object $E^l_A(U_i)$, and a neighbourhood of the last coupon contains a graph representing the following morphism.

\begin{figure}[h]
	\centering
	\begin{picture}(80,100)(0,0)
	\put(0,0)	{\includegraphics[scale=0.3]{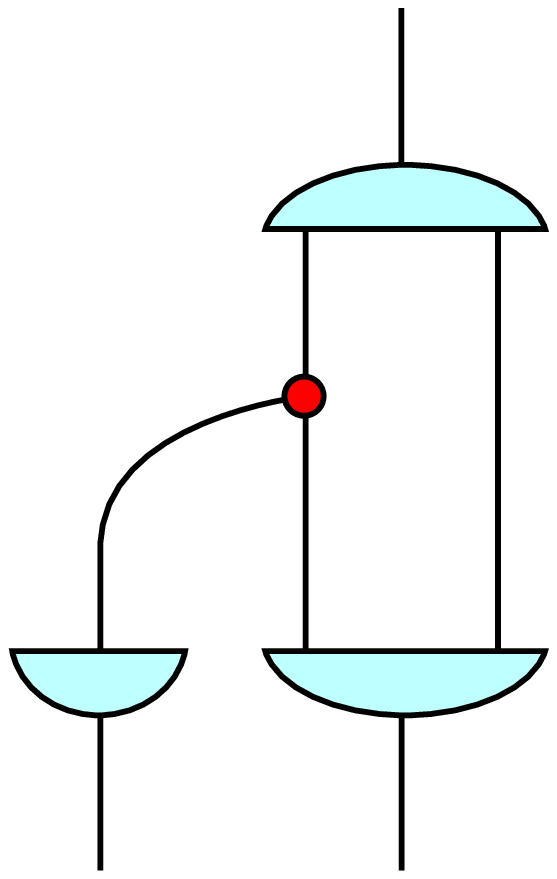}}
	\put(0,-8)	{\scriptsize $C_l(A)$}
	\put(30,-8)	{\scriptsize $E^l_A(U_i)$}
	\put(15,80)	{\scriptsize $E^l_A(U_i)$}
	\end{picture}
	\be{repmorph}
	\ee
	\label{fig:repmorph}
\end{figure}

We are now ready to prove the theorem.\\

\noindent Proof of Theorem ~\ref{thm:reducibility}:\\
Without loss of generality we can assume that $A$ is haploid.
Since $A$ gives a non-trivial torus partition function, according to \eqref{eq:modinv} there exists a simple object $U_i$ s.t. $E^l_A(U_i)\ncong U_i$ and is also not the zero-object. Take $v^g_i\in\bl(\extS_g)$ as above corresponding to this object.
If there does not exist any $\zeta\in\C$ such that $P_g(v^g_i)=\zeta v^g_i$ then, since $P_g$ commutes with $\rho$, Schur's lemma implies that $\rho$ is reducible.

The element $P_g(v^g_i)$ is represented by a ribbon graph in a handle body obtained by gluing $\eCob[\extS_g,A,T]$ to $\eCob^g_{U_i}$. Represent the morphisms $P^l_A$ and $P^l_A(U_i)$ as $e_{C_l(A)\prec A}\circ r_{C_l(A)\prec A}$ and $e_{E^l_A(U_i)\prec A}\circ r_{E^l_A(U_i)\prec A}$ respectively, and slide the corresponding coupons towards the vertices of the ribbon graph.
For $g=1$, this leaves a single ribbon labelled by $E^l_A(U_i)$ with the endomorphism $r_{E^l_A(U_i)\prec A}\circ e_{E^l_A(U_i)\prec A}=\id_{E^l_A(U_i)}$, and we have $P_1(v^1_i)=v^1_{E^l_A(U_i)}$ which by assumption is 
not proportional to $v^1_i$.\\
For $g>1$ we get a ribbon graph where all but one of the ribbons are labelled by the left center $C_l(A)$, and all but one of the coupons are labelled by either the multiplication $m_{C_l(A)}$ or comultiplication $\Delta_{C_l(A)}$. The remaining ribbon and coupon are labelled by the object $E^l_A(U_i)$ resp. the morphism \eqref{eq:repmorph}. Decompose $C_l(A)$ in simple subobjects, and $m_{C_l(A)}$ and $\Delta_{C_l(A)}$ into the corresponding components. Since $A$ is haploid, so is $C_l(A)$, and the decomposition therefore contains a {\em unique} component where all but one of the objects is $\one$, and all but one of the morphisms are either $(m_{C_l(A)}) _{00}^{\phantom{00}0}\id_\one$ or $(\Delta_{C_l(A)}) _{0}^{\phantom{0}00}\id_\one$. According to \eqref{eq:unitcomp}, both of these morphisms are exactly $\id_\one$. The corresponding component of the morphism \eqref{eq:repmorph} reduces to
\begin{figure}[h]
	\centering
	\includegraphics[scale=0.3]{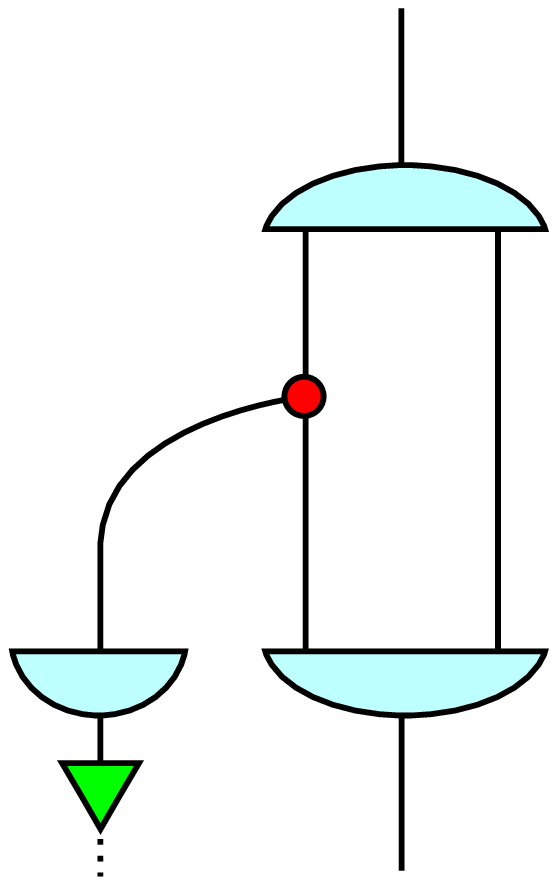}
		\put(-40,-8)	{\scriptsize $\one$}
	\put(-20,-8)	{\scriptsize $E^l_A(U_i)$}
	\put(-30,80)	{\scriptsize $E^l_A(U_i)$}
	\be{repmcomp}
	\ee
\end{figure}

\noindent Since both $A$ and $C_l(A)$ are haploid, it follows that
$e_{C_l(A)\prec A}\circ e_{\one\prec C_l(A)}=\xi\eta$ for some $\xi\in\C^\times$. Left unitality of $A$, together with the identity $r_{E^l_A(U_i)\prec A}\circ e_{E^l_A(U_i)\prec A}=\id_{E^l_A(U_i)}$, implies that \eqref{eq:repmcomp} is just $\xi\id_{E^l_A(U_i)}$.
The induced decomposition of $P_g(v^g_i)$ thus contains a term $\xi v^g_{E^l_A(U_i)}$, which by assumption is not proportional to $v^g_i$. The uniqueness of the component considered above in the decomposition implies that  $P_g(v^g_i)-\xi v^g_{E^l_A(U_i)}$ is either zero or linearly independent of $\{v^g_{E^l_A(U_i)},v^g_i\}$, and $P_g(v^g_i)$ is therefore not proportional to $v^g_i$.
We have thus shown that $P_g$ is not proportional to the identity.$\quad \Box$

\subsection{Examples}
There is a rather extensive list of known modular invariants (coefficients $Z_{ij}$ related to some modular category) in the litterature, so there is potentially a similar list of TQFT's with reducible mapping class group representations according to theorem~\ref{thm:reducibility}. One aspect, however, keeps us from immediately drawing this conclusion. The theorem states that the modular invariant $Z_{ij}$ must arise from a symmetric special Frobenius algebra, which according to~\cite{FjFRS1,FjFRS2} is equivalent to saying that the coefficients $Z_{ij}$ arise as the coefficients of the torus partition function in some (open/closed) rational conformal field theory. This is certainly not known for a large set of modular invariants.
As mentioned in the introduction, there is a large class of modular invariants that are known to be realised in terms of symmetric special Frobenius algebras, the simple curent invariants. It's important to point out that there are slightly different notions of simple current invariants in the litterature, and the relevant notion for this discussion is that presented (and completely classified) in~\cite{KrS}. It was shown in~\cite{FRS3} that there is a symmetric special Frobenius algebra realising every such simple current invariant.
We briefly discuss the cases listed in the corollary to theorem~\ref{thm:reducibility}.

\noindent
For the $SU(2)$ theories we refer to the next section. For $SU(3)$, Ocneanu has announced~\cite{Ocneanu} that all modular invariants (which are completely classified) have been realised in terms of Q-systems in the subfactor language. It is interesting to note that there exist non-trivial torus partition functions for all levels greater than or equal to $3$ in the $SU(3)$ case~\cite{Gannon}.
The series of $SU(N)$ models appear in the corollary due to the appearance of simple current invariants and the result discussed above. In addition Q-systems have been constructed corresponding to simple current invariants for $SU(N)$, see~\cite{BE2}.

\noindent
Apart from these examples one should mention that all modular invariants of the unitary minimal models have been realised~\cite{BE2}, as have those for the rational extended $U(1)$ theories~\cite{FRS1}. In addition, a number of modular invariants in theories corresponding to quantum doubles of the finite groups $S_3$ and $\Z/n\Z$ are realised in~\cite{EvPinto}.\\

\noindent
Apart from the models listed above, all modular invariants have been found for the $(A_1+A_1)^{(1)}$ series, the $(U(1)+\ldots +U(1))^{(1)}$ series, and for $A_r^{(1)}$, $B_r^{(1)}$, and $D_r^{(1)}$ for all ranks and levels $k\leq 3$. In the light of theorem~\ref{thm:reducibility}, quantum representations in TQFT's based on these affine algebras are expected to be generically reducible for all genus.

\subsection{The $\mathfrak{su}(2)$ Case}\label{ssec:su2case}
The structures that allow us to efficiently investigate how the representations decompose in explicit examples are those of direct sum and tensor product of algebras in \cC.
\begin{definition}
	Let $A\equiv(A,m_A,\eta_A)$ and $B\equiv(B,m_m,\eta_B)$ be algebras in \cC.
	\begin{itemize}
	\item The tensor product $A\boxtimes B$ of $A$ and $B$ is the algebra $(A\otimes B,(m_A\otimes m_B)\circ(\id_A\otimes c_{B,A}\otimes\id_B),\eta_A\otimes\eta_B)$
	\item The direct sum $A\boxplus B$ of $A$ and $B$ is the algebra $(A\boxplus B,m_{A\oplus B},\eta_{A\boxplus B})$, where $m_{A\boxplus B}=e_{A\prec A\oplus B}\circ m_A\circ(r_{A\prec A\oplus B}\oti r_{A\prec A\oplus B}) + e_{B\prec A\oplus B}\circ m_B\circ(r_{B\prec A\oplus B}\oti r_{B\prec A\oplus B})$ and $\eta_{A\boxplus B}=e_{A\prec A\oplus B}\circ \eta_A\circ r_{A\prec A\oplus B} + e_{B\prec A\oplus B}\circ \eta_B\circ r_{B\prec A\oplus B}$
	\end{itemize}
\end{definition}
It is straightforward to check that the morphisms satisfy the required properties of associativity and unitality. The corresponding definitions for coalgebras are analogous, with the coproduct of $A\boxtimes B$ being $(\id_A\otimes c^{-1}_{A,B}\otimes\id_B)\circ(\Delta_A\otimes\Delta_B)$.
With these definitions for algebras and coalgebras, the tensor product and direct sum of two Frobenius algebras is again Frobenius, and the same holds for the properties (normalized) special and symmetric.
\begin{remark}
Note that the braiding gives two possible structures of multiplication and comultiplication on the object $A\otimes B$, we have simply chosen one for each such that the Frobenius and special properties are preserved.
\end{remark}
The following proposition follows from a straightforward generalisation of proposition 5.3 of \cite{FRS1}.
\begin{proposition}\label{prop:Pmultadd}~\\
	\begin{itemize}
		\item[a)] $P_g[A\boxtimes B] = P_g[A]\circ P_g[B]$
		\item[b)] $P_g[A\boxplus B] = P_g[A] + P_g[B]$
	\end{itemize}
\end{proposition}
Equation 3.37 of \cite{FrFRSduality} can be specialized to the following
\begin{proposition}\label{prop:PMorita}
Let $A$ and $B$ be two Morita equivalent symmetric special Frobenius algebras in \cC, then
$$P_g[A]=\gamma^{-\chi(\extS_g)/2} P_g[B]$$ where $\gamma=\dim(B)/\dim(A)$. In particular we have $Z_{ij}(A)=Z_{ij}(B)$.
\end{proposition}
The next result has been announced in~\cite{FRS1}.
\begin{proposition}\label{prop:Moritastability}
Let $\sim$ denote Morita equivalence, and let $A$, $A'$, $B$, $B'$ be algebras in \cC{} such that $A\sim A'$, $B\sim B'$. Then
	\begin{itemize}
		\item[a)] $A\boxtimes B \sim A'\boxtimes B'$
		\item[b)] $A\boxplus B \sim A'\boxplus B'$
	\end{itemize}
\end{proposition}
\begin{proof}
Follows straightforwardly from the definition of a Morita context in terms of interpolating bimodules.
\end{proof}

~\\
If $[A]$ indicates the Morita class of $A$ we have in other words that $[A\boxtimes B]=[A'\boxtimes B']$ and $[A\boxplus B]=[A'\boxplus B']$.
The properties simple, special, symmetric and Frobenius are stable under Morita equivalence, as stated in
\begin{proposition}
	Let $A$ be a simple symmetric special Frobenius algebra in \cC, then any algebra in $[A]$ is also simple symmetric special Frobenius.
\end{proposition}
\begin{proof}
Follows from Proposition $2.13$ in \cite{FRS2}
\end{proof}

~\\
Finally, we recall a corollary from \cite{Ostrik}
\begin{proposition}
Any simple symmetric special Frobenius algebra in \cC{} is Morita equivalent to a Haploid algebra.
\end{proposition}

Using these results
 we introduce a structure of a unital rig on the set of Morita classes of symmetric special Frobenius algebras in \cC.
\begin{definition}
The {\em Frobenius rig} $\F_\cC$ of \cC{} is the rig generated by the set $$\{[A]|\textrm{$A$ is a simple special Frobenius algebra}\}$$
with
\begin{itemize}
\item addition: $[A]+[B]\equiv[A\boxplus B]$
\item multiplication: $[A]\times[B]\equiv[A\boxtimes B]$
\item unit: $[\one]$
\end{itemize}
\end{definition}

\begin{remark}
~\\[-2ex]
	\begin{itemize}
	\item[(i)] A closely related concept for subfactors was introduced in \cite{EvPinto} under the name {\em fusion of modular invariants}. The name makes sense since, as we shall see later, at least in the $\mathfrak{su}(2)$ case, all the information of this rig can be found in the coefficients $Z_{ij}$.
	\item[(ii)] Checking the necessary properties of the addition and multiplication requires checking
	that the
	bicategory $\mathcal{B}im(\cC)$ of symmetric special Frobenius algebras, with morphism categories the categories of bimodules,
	is a {\em monoidal} semisimple bicategory. We refer to a later publication for the proof~\cite{FS}.
	\item[(iii)] By the Grothendieck construction, the rig $\F_\cC$ becomes a ring. This ring has some right to be called the Grothendieck ring of $\mathcal{B}im(\cC)$, although that notion for a bicategory has been used for a different structure.
	\end{itemize}
\end{remark}
If it happens that the modular invariant torus partition functions, $Z_{ij}(A)$, with $Z_{00}(A)=1$ are in bijection with the Morita classes of simple symmetric Frobenius algebras in \cC{}, then from the propositions above it follows that the matrices $Z_{ij}(A)$ for different $A$ form a faithful matrix representation of $\F_\cC$.\\[1ex]

\noindent
By classifying module categories over the categories $\cC_k$, the Morita classes of simple symmetric special Frobenius algebras were classified in \cite{KO, Ostrik}. The result is in precise agreement with the ADE classification of modular invariant torus partition functions in \cite{CaItZu}, so $\F_{\cC_k}$ can be determined from $Z_{ij}(A)$ for different $A$. Denote the Morita classes of algebras of types $A$, $D$, $E$ by $[A]$, $[D]$, $[E]$. The rig $\F_{\cC_k}$ is then commutative, has $[A]$ as unit,
 and takes the following form~\cite{EvPinto}.\\[1ex]

\begin{itemize}
	\item $k=0$ mod $4$: $[D]\times[D]  =  2[D]$
	\item $k=2$ mod $4$: $[D]\times[D]  =  [A]$
	\item $k=10$: $[D]\times[E]  =  [E]$, $[E]\times[E] = 2[E]$
	\item $k=16$: $[D]\times[E]  =  2[E]$, $[E]\times[E] = [D]+[E]$
	\item $k=28$: $[D]\times[E]  =  2[E]$, $[E]\times[E] = 4[E]$
\end{itemize}

The propositions above, together with the structure of $\F_{\cC_k}$, prove the following
\begin{proposition}\label{prop:PnFnfusion}
In the category $\cC_k$, let $D$and $E$ denote the simple symmetric special Frobenius algebras corresponding to the $D$-type resp. $E$-type modular invariant torus partition functions as given in table~\ref{tab:ADEclass}. Let $d_D$ and $d_E$ be the (non-zero) quantum dimensions of the underlying objects.
For any genus $g$, the endomorphisms $P_g$ satisfy
\begin{align}
	P_g[D]\circ P_g[D]& =  \left(2d_D^{-1}\right)^{-\chi(\extS_g)/2}2P_g[D] & \text{for }k=0\ \mathrm{mod}\ 4\\
	P_g[D]\circ P_g[D]& =  d_D^{\chi(\extS_g)}\id_{\bl(\extS_g)} & \text{for } k=2\ \mathrm{mod}\ 4\\
	\begin{split}
		P_g[D]\circ P_g[E]& =d_D^{\chi(\extS_g)/2}P_g[E]\\
		P_g[E]\circ P_g[E]& = \left(2d_E^{-1}\right)^{-\chi(\extS_g)/2}2P_g[E]
	\end{split}
	& \text{for } k=10\\
	\begin{split}
		P_g[D]\circ P_g[E]& = \left(2d_D^{-1}\right)^{-\chi(\extS_g)/2}2P_g[E]\\
		P_g[E]\circ P_g[E]&= \left[(d_D+d_E)d_E^{-2}\right]^{-\chi(\extS_g)/2}\left(P_g[D] + P_g[E]\right)
	\end{split}
	& \text{for } k=16\\
	\begin{split}
		P_g[D]\circ P_g[E]& = \left(2d_E^{-1}\right)^{-\chi(\extS_g)/2}2P_g[E]\\
		P_g[E]\circ P_g[E]& = \left(4d_E^{-1}\right)^{-\chi(\extS_g)/2}4P_g[E]
	\end{split}
	& \text{for } k=28
\end{align}
\end{proposition}
~\\

When $k=0$ mod $4$, the endomorphisms
\be{Proj0}
	\begin{split}
		\Pi^{g,k,D}_{+}& \equiv\frac{\left(2d_D^{-1}\right)^{\chi(\extS_g)/2}}{2}P_g[D]\\
		\Pi^{g,k,D}_{-}& \equiv \id_{\bl(\extS_g)}-P_g^{D+}
	\end{split}
\ee
are idempotent.
For $k=2$ mod $4$, corresponding idempotents are defined by
\be{Proj2}
	\Pi^{g,k,D}_\pm\equiv\frac{1}{2}\left(\id_{\bl(\extS_g)}\pm d_D^{-\chi(\extS_g)/2}P_g[D]\right).
\ee
 Furthermore, since $P_g[D]$ is not proportional to the identity, the rank of $\Pi^{g,k,D}_{\pm}$ does not vanish for any $g\geq 1$ and $k\in 2\Z_+$.\\
 For $k=10, 16, 28$ there are additional idempotents given by
 \begin{align}
 	\begin{split}
 		\Pi^{g,10,E}_+ & \equiv \frac{\left( 2d_E^{-1}\right)^{\chi(\extS_g)/2}}{2}P_g[E]\\
		\Pi^{g,10,E}_- & \equiv \id_{\bl(\extS_g)}-\Pi^{g,10,E}_+
	\end{split}
	& \text{for $k=10$}\\
	\begin{split}
		\Pi^{g,16,E}_+ & \equiv \frac{\left(2d_D^{-1}\right)^{\chi(\extS_g)/2}}{\sqrt{16\gamma+\gamma^2}}\left(\left(\sqrt{16\gamma+\gamma^2}-\gamma/4\right)\Pi^{g,16,D}_++P_g[E]\right)\\
		\Pi^{g,16,E}_- & \equiv  \id_{\bl(\extS_g)}-\Pi^{g,16,E}_+
	\end{split}
	& \text{for $k=16$}\\
	\begin{split}
		\Pi^{g,28,E}_+ & \equiv \frac{\left( 4d_E^{-1}\right)^{\chi(\extS_g)/2}}{4}P_g[E]\\
		\Pi^{g,28,E}_- & \equiv  \id_{\bl(\extS_g)}-\Pi^{g,28,E}_+
	\end{split}
	& \text{for $k=28$}
\end{align}
where $\gamma=\left[(d_D+d_E)d_E^{-2}\right]^{-\chi(\extS_g)/2}$.
Similarly as above we conclude that $\Pi^{g,10,E}_{\pm}$ and $\Pi^{g,28,E}_\pm$ do not have vanishing rank. An explicit check confirms that also the ranks of $\Pi^{g,16,E}_\pm$ are non-vanishing.
Using proposition~\ref{prop:PnFnfusion} it is easily shown that
\begin{align}
	\Pi^{g,k,D}_+\Pi^{g,k,E}_+ &= \Pi^{g,k,E}
\end{align}
for $k=10, 16, 28$.
Furthermore, explicit calculations confirm that $\Pi^{g,k,D}_+\Pi^{g,k,E}_-$ has non-vanishing rank for the same values of $k$.

Combining these results, we have shown
theorem \ref{thm:su2}.

\begin{remark}
	It is true that theorem~\ref{thm:reducibility} and the methods used in the proof already imply that the representations considered in theorem~\ref{thm:su2} decompose in direct sums of two sub-representations. Theorem~\ref{thm:su2} contains more information than this, however. First, it shows the non-trivial decomposition in three sub-representations in the presence of the exceptional series of algebras. Second,  the projectors allow us to determine the dimensions of the subrepresentations. It is for instance straightforward to determine the dimensions of the sub-representations. As a simple example we get $\dim(V^{1,4n,D}_+)=n+1$, $\dim(V^{1,4n,D}_-)=3n$, $\dim(V^{1,4n+2,D}_+)=n+1$, $\dim(V^{1,4n,D}_+)=3n+2$, $\dim(V^{1,10,E}_-)=5$, $\dim(W^{1,10})=3$.
\end{remark}

\end{document}